\documentclass[11pt,reqno]{amsart}
\usepackage[a4paper,margin=27mm,headheight=14pt]{geometry}
\usepackage[T1]{fontenc}
\usepackage{lmodern,microtype}
\usepackage{amsmath,amssymb,amsthm,booktabs}
\usepackage[hidelinks,pdfencoding=auto]{hyperref}
\hypersetup{pdftitle={Boundary Rigidity of Line-Bundle Shilov--Poisson Transforms on Bounded Symmetric Domains}}
\allowdisplaybreaks[2]
\numberwithin{equation}{section}
\newtheorem{theorem}{Theorem}[section]
\newtheorem{proposition}[theorem]{Proposition}
\newtheorem{lemma}[theorem]{Lemma}
\newtheorem{remark}[theorem]{Remark}

\newtheorem{definition}[theorem]{Definition}
\newtheorem{corollary}[theorem]{Corollary}

\title[Boundary Rigidity of Line-Bundle Shilov--Poisson Transforms]{Boundary Rigidity of Line-Bundle Shilov--Poisson Transforms on Bounded Symmetric Domains}

\date{}

\author[Chen]{Ren-Yu Chen}
\address{School of Mathematics, Tianjin University, Tianjin 300350, P.R. China}
\email{chenry@tju.edu.cn}

\author[Liu]{Su Liu}
\address{School of Mathematics, Tianjin University, Tianjin 300350, P.R. China}
\email{liusu1022@tju.edu.cn}

\subjclass[2020]{Primary 32M15; Secondary 32A25, 32A50, 43A85}

\keywords{Shilov--Poisson transforms, bounded symmetric domains,
Graham type theorem, homogeneous line bundles,
covariant differential operators}

\begin{document}
\begin{abstract}
Let $\Omega=G/K$ be an irreducible bounded symmetric domain other
than the unit disk, and let $S$ be its Shilov boundary. We study the
two-parameter family of Shilov--Poisson transforms between homogeneous
line bundles over $S$ and $\Omega$.
For a transform $F$ of $L^2$ boundary data, we determine what finite
Euclidean regularity up to the full boundary $\partial\Omega$ forces
in the interior.  Outside five explicit exceptional parameter families,
sufficiently high finite regularity forces $F=0$.  In each exceptional
family, the conclusion is an interior covariant differential equation.
 In the scalar
Poisson--Szeg\H{o} case, it yields a finite-order Graham type theorem:
the stated boundary regularity forces a Bergman-harmonic function
to be pluriharmonic.
\end{abstract}

\maketitle


\section{Introduction}

On the unit ball $\mathbb B^n$, $n\ge2$, Graham \cite{Graham}
proved that a Bergman-harmonic function
$u\in C^n(\overline{\mathbb B^n})$ is pluriharmonic.
Every Bergman-harmonic function on the ball that is continuous up to the
boundary is the Poisson--Szeg\H{o} transform of its boundary values
\cite[pp.~532--533]{Inoue}. Thus Graham's theorem asks what boundary regularity
of a Poisson transform forces in the interior.  We take up this question
for Shilov--Poisson transforms on irreducible bounded symmetric domains
other than the unit disk, allowing the bundle and transform parameters
to vary. This brings together two lines of work: boundary rigidity
after Graham, and the theory of Poisson transforms and invariant Hua
operators.

The boundary-rigidity literature includes work by Graham--Lee \cite{GrahamLee} on degenerate
Laplacians on strictly pseudoconvex domains.
S.-Y.~Li and collaborators \cite{LiNi,LiSimon,LiWei} studied boundary behavior and rigidity
for harmonic maps and functions in Bergman-type metrics
.  On bounded symmetric domains,
Bonami--Buraczewski--Damek--Hulanicki--Jaming \cite{BBDHJ} studied maximal boundary
regularity for bounded Hua-harmonic functions on tube domains
.  Chen--Li \cite{ChenLi} proved a Graham-type pluriharmonicity theorem
on several types of classical domains, and Xiao \cite{Xiao} studied
boundary pluriharmonicity and related rigidity questions there
.  Related work on the ball includes Ahern--Bruna--Cascante
\cite{ABC} and Q.~Zhang \cite{QZhang17}.

The Poisson-transform side begins with Hua.  His work\cite{HuaBook} on harmonic functions and Poisson
formulas on the classical domains helped establish what are now called
Hua systems ; see also Hua--Lu's monograph \cite{HuaLook}.
Kor\'anyi \cite{Koranyi65} extended the Poisson construction to general bounded symmetric
domains, and Johnson--Kor\'anyi  \cite{JohnsonKoranyi} introduced the invariant
Hua operator on tube domains.
Kor\'anyi--Malliavin \cite{KoranyiMalliavin}, Berline--Vergne
\cite{BerlineVergne}, and Lassalle \cite{Lassalle} developed the equations
and their Poisson formulas further.  Boussejra
\cite{Boussejra03,Boussejra06,Boussejra11} and Boussejra--Koufany
\cite{BK07} studied boundary representations.  The theory for varying
parameters and homogeneous line bundles was developed by Shimeno
\cite{Shimeno96,Shimeno}, Koufany--Zhang \cite{KZ06,KZ}, and Kor\'anyi
\cite{Koranyi}.  Related work on line bundles includes Zhang
\cite{ZhangBerezin00}, Boussejra--Ourchane \cite{BO}, and
El Wassouli--Oukacha \cite{EWO}.

The two lines meet at the Shilov--Poisson transform: it takes data on the
Shilov boundary to sections satisfying invariant differential equations.
We ask what further restrictions follow when such a section has finite
Euclidean regularity up to the full boundary $\partial\Omega$.  The answer
depends on the parameters.  Outside five exceptional families,
sufficiently high but finite regularity forces the transform to vanish.
In each exceptional family, the interior conclusion is a covariant
differential equation.  In two families, the equation holds without any
boundary regularity assumption.  Our classification covers all parameter values on these domains,
including domains of exceptional type.

Throughout the paper, let $\Omega=G/K\subset V$ be an irreducible bounded
symmetric domain other than the unit disk. Write $n$ and $r$ for its
complex dimension and rank, and let $S$ be its Shilov boundary.
If $a,b$ are its characteristic multiplicities, then
$\frac nr=1+b+\frac a2(r-1)$.
Let
$\mathbb N=\{1,2,\ldots\}$ and $\mathbb N_0=\mathbb N\cup\{0\}$.  

Following
Kor\'anyi \cite{Koranyi}, for $\nu,s\in\mathbb C$ let $L^\Omega_\nu$ and
$L^S_{\nu;s}$ be the corresponding homogeneous line bundles over $\Omega$
and $S$, and write $$P_{\nu,s}:L^2(S,L^S_{\nu;s})\longrightarrow C^\infty(\Omega,L^\Omega_\nu)$$
for the Shilov--Poisson transform.  See Section~\ref{5} for details.

Boundary regularity below always means Euclidean regularity of the scalar
coefficient in the canonical trivialization.  For $m\in \mathbb{N}_0$ and $\varepsilon\in (0,1]$,  the notation $C^m(\overline\Omega,L^\Omega_\nu)$ and
$C^{m,\varepsilon}(\overline\Omega,L^\Omega_\nu)$ requires the scalar
coefficient to extend to a Euclidean neighborhood of $\overline\Omega$
with $C^m$ and $C^{m,\varepsilon}$ regularity, respectively. For $q>0$, $\mathcal R^q$ means $C^q$ when
$q\in\mathbb N$, and otherwise the union of
$C^{\lfloor q\rfloor,\varepsilon}$ over
$q-\lfloor q\rfloor<\varepsilon\le1$.  Section~\ref{5} gives the precise
definitions and also defines the order-$(L+1)$ covariant differential
operator $\mathbb D_{L,j}$ used below.
In the third row below, $\bar\partial$ is the Dolbeault operator
applied to $\mathbb D_{-\nu,0}F$. In the Harish--Chandra
trivialization, it acts componentwise.

\begin{theorem}\label{3}
Let $\nu,s\in\mathbb C$, let $f\in L^2(S,L^S_{\nu;s})$, and put
$F=P_{\nu,s}f$.  Then

\begin{center}
\small
\renewcommand{\arraystretch}{1.5}
\begin{tabular}{@{}p{0.34\textwidth}p{0.60\textwidth}@{}}
\hline
\textbf{Parameter} & \textbf{Conclusion}\\
\hline

\(\displaystyle s+\nu=\frac{2n}{r}\)
&
\(F\) is holomorphic.
\\
\hline

\(\displaystyle s=\nu,\quad
\nu\ne\frac nr,\quad \nu\notin-\mathbb N_0\)
&
\(\displaystyle
F\in
C^{\max\{0,\lceil\frac nr-\operatorname{Re}\nu\rceil\}}
(\overline\Omega,L^\Omega_\nu)
\Longrightarrow
F\text{ is holomorphic}.
\)
\\
\hline

\(\displaystyle s=\nu\in-\mathbb N_0\)
&
\(\displaystyle
F\in
\begin{cases}
\mathcal R^{\frac nr-\nu}(\overline\Omega,L^\Omega_\nu),
&\Omega\text{ is of tube type},\\
C^{b+1-\nu}(\overline\Omega,L^\Omega_\nu),
&\Omega\text{ is non-tube}
\end{cases}
\Longrightarrow
\bar\partial\mathbb D_{-\nu,0}F=0.
\)
\\
\hline

\(\displaystyle
\frac nr-\frac{s+\nu}{2}\in\mathbb N,\quad
\frac nr-\frac{s-\nu}{2}\in-\mathbb N_0
\)
&
\(\displaystyle
\mathbb D_{
\frac{s-\nu}{2}-\frac nr,\,
\frac nr-\frac{s+\nu}{2}}F=0.
\)
\\
\hline

\(\displaystyle
-\frac{s+\nu}{2}\in\mathbb N_0,\quad
\frac{s-\nu}{2}\in\mathbb N,\quad
\frac nr-\frac{s-\nu}{2}\notin-\mathbb N_0
\)
&
\(\displaystyle
F\in
\mathcal R^{\frac nr-\frac{s+\nu}{2}}
(\overline\Omega,L^\Omega_\nu)
\Longrightarrow
\mathbb D_{
-\frac{s+\nu}{2},\,
\frac{s-\nu}{2}}F=0.
\)
\\
\hline

All remaining parameters
&
There exists $m\in\mathbb N_0$ such that
\(\displaystyle
F\in C^m(\overline\Omega,L^\Omega_\nu)
\Longrightarrow F=0.
\)
\\
\hline
\end{tabular}
\end{center}
\end{theorem}

At $(\nu,s)=(0,0)$, both line bundles are trivial and $P_{0,0}$ is the
Poisson--Szeg\H{o} transform. This case of Theorem~\ref{3} gives
the following result for Bergman-harmonic functions.

\begin{corollary}[Graham type Theorem]
\label{4} Let $F\in C^{2}(\Omega)$ be Bergman-harmonic. If
\[
F\in%
\begin{cases}
\mathcal{R}^{\frac nr}(\overline\Omega), & \Omega\text{ is of tube
type},\\[1mm]%
C^{b+1}(\overline\Omega), & \Omega\text{ is non-tube},
\end{cases}
\]
then $F$ is pluriharmonic.
\end{corollary}

On the unit ball $\mathbb B^n$, $n\ge2$, Corollary~\ref{4}
is Graham's theorem.  On tube-type
domains, its pluriharmonicity conclusion overlaps with the earlier
result for bounded Hua-harmonic functions in \cite{BBDHJ}, although the
boundary regularity assumptions differ.  In higher rank,
Corollary~\ref{4} covers every irreducible bounded symmetric domain,
including the exceptional domains; compare \cite{ChenLi,Xiao} for the
classical domains.  The required regularity has finite order and depends
only on the root multiplicities of $\Omega$.

The proof first uses radial Poisson transforms on minimal disks to
restrict the possible parameters. The case $s=\nu\in-\mathbb N_0$
requires separate arguments for tube and non-tube domains.

Section~\ref{5} recalls the Poisson transform and the basic notation.
Section~\ref{39} determines the possible parameter sets.
Sections~\ref{63} and~\ref{69} treat the tube and non-tube cases.
Section~\ref{87} proves Theorem~\ref{3} and Corollary~\ref{4}.
Appendix~\ref{90} completes the proof of Proposition~\ref{84}.

\section{The Poisson transform and basic properties}

\label{5}

We retain $n,r,a,b$ from the Introduction and let \(p_\Omega=(r-1)a+b+2\) be
the genus of $\Omega$.

\subsection{Line bundles and the scalar kernel}

The Poisson transform in Theorem~\ref{3} acts between line bundles.
We first describe these bundles and their sections, then pass to the
scalar kernel used in the proofs.

Let $G=\operatorname{Aut}(\Omega)^\circ$ be the identity component of the
biholomorphic automorphism group of $\Omega$. Let $K$ be the subgroup
fixing $0$, so that $\Omega=G/K$. Write $\widetilde G$ for the universal
covering group of $G$, let $\pi:\widetilde G\to G$ be the covering map,
and put $\widetilde K=\pi^{-1}(K)$. Fix $e_0\in S$, let $P_1$ be the
subgroup of $G$ fixing $e_0$, and put
$\widetilde P_1=\pi^{-1}(P_1)$. Since $G$ acts transitively on $S$,
we have $S\simeq G/P_1\simeq\widetilde G/\widetilde P_1$.
For $\widetilde g\in\widetilde G$, write $g=\pi(\widetilde g)$ and
\[
J_{\widetilde g}(z)=\det_{\mathbb C}g'(z),\qquad z\in\Omega.
\]
For each $\lambda\in\mathbb C$, the lift $\widetilde g$ specifies
the branch of $J_{\widetilde g}^{\lambda}$ obtained by continuation
from the identity. Lemma~\ref{28}
shows that this branch extends holomorphically across
$\overline\Omega$.

For $\nu,s\in\mathbb C$, define
\[
\begin{aligned}
\tau_\nu(\widetilde k)
&=J_{\widetilde k^{-1}}(0)^{\nu/p_\Omega},
&&\widetilde k\in\widetilde K,\\
\chi_{\nu,s}(\widetilde p)
&=J_{\widetilde p}(e_0)^{-\frac{s+\nu}{2p_\Omega}}
  \overline{J_{\widetilde p}(e_0)}^{-\frac{s-\nu}{2p_\Omega}},
&&\widetilde p\in\widetilde P_1.
\end{aligned}
\]
By the Jacobian cocycle identity \eqref{23}, these are characters, since elements
of $\widetilde K$ fix $0$ and elements of $\widetilde P_1$ fix $e_0$.
They define the line bundles
\[
L^\Omega_\nu
=\widetilde G\times_{\widetilde K}\mathbb C_{\tau_\nu}
\longrightarrow\Omega,
\qquad
L^S_{\nu;s}
=\widetilde G\times_{\widetilde P_1}\mathbb C_{\chi_{\nu,s}}
\longrightarrow S.
\]
The corresponding equivalence relations are
\[
\begin{aligned}
(\widetilde g,\zeta)
&\sim(\widetilde g\widetilde k,\tau_\nu(\widetilde k)^{-1}\zeta),
&&\widetilde k\in\widetilde K,\\
(\widetilde g,\zeta)
&\sim(\widetilde g\widetilde p,\chi_{\nu,s}(\widetilde p)^{-1}\zeta),
&&\widetilde p\in\widetilde P_1.
\end{aligned}
\]

A section $\sigma$ of $L^\Omega_\nu$ is represented by a function
$\sigma^\sharp$ on $\widetilde G$ through
$\sigma(g0)=[\widetilde g,\sigma^\sharp(\widetilde g)]$.
Likewise, a section $\varphi$ of $L^S_{\nu;s}$ is represented by
$\varphi^\sharp$ through
$\varphi(ge_0)=[\widetilde g,\varphi^\sharp(\widetilde g)]$.
They satisfy
\[
\begin{aligned}
\sigma^\sharp(\widetilde g\widetilde k)
&=\tau_\nu(\widetilde k)^{-1}\sigma^\sharp(\widetilde g),
&&\widetilde k\in\widetilde K,\\
\varphi^\sharp(\widetilde g\widetilde p)
&=\chi_{\nu,s}(\widetilde p)^{-1}\varphi^\sharp(\widetilde g),
&&\widetilde p\in\widetilde P_1.
\end{aligned}
\]
Using this equivariance, we obtain scalar functions $F$ on $\Omega$
and $f$ on $S$ by
\[
\begin{aligned}
F(g0)
&=J_{\widetilde g}(0)^{-\nu/p_\Omega}
  \sigma^\sharp(\widetilde g),\\
f(ge_0)
&=J_{\widetilde g}(e_0)^{-\frac{s+\nu}{2p_\Omega}}
  \overline{J_{\widetilde g}(e_0)}^{-\frac{s-\nu}{2p_\Omega}}
  \varphi^\sharp(\widetilde g).
\end{aligned}
\]
By \eqref{23} and the equivariance above, both values are independent
of the chosen lift $\widetilde g$.
Conversely, $F$ and $f$ determine $\sigma$ and $\varphi$. We use these
scalar representatives throughout the paper. In particular, a boundary
section belongs to $L^2(S,L^S_{\nu;s})$ precisely when its scalar
representative $f$ belongs to $L^2(S)$.

Left multiplication on the first factor gives natural
$\widetilde G$-actions on the two bundles. On lifted sections this is
left translation. For the scalar representative of a section of
$L^\Omega_\nu$, the action becomes
\begin{equation}
\label{6}
(T_\nu(\widetilde g)F)(z)
=J_{\widetilde g^{-1}}(z)^{\nu/p_\Omega}F(g^{-1}z).
\end{equation}

The Poisson transform sends a boundary section $\varphi$ to a section
of $L^\Omega_\nu$ by averaging its lift. In the notation above,
\cite[formula~(2.1)]{Koranyi} gives
\begin{equation}
\label{7}
(P_{\nu,s}\varphi)^\sharp(\widetilde g)
=\int_{\widetilde K/Z(\widetilde G)}
 \tau_\nu(\widetilde k)
 \varphi^\sharp(\widetilde g\widetilde k)\,d\widetilde k,
\qquad \widetilde g\in\widetilde G,
\end{equation}
where $Z(\widetilde G)$ is the center of $\widetilde G$, and
$d\widetilde k$ is the normalized Haar measure on
$\widetilde K/Z(\widetilde G)$.

To write the transform in scalar form, let $h(z,w)$ be the generic
norm of $\Omega$, normalized by $h(z,0)=h(0,w)=1$. It is polynomial
in $z$ and $\bar w$, holomorphic in $z$, antiholomorphic in $w$,
and satisfies $h(z,w)=\overline{h(w,z)}$. The Bergman kernel of
$\Omega$ is a positive constant multiple of $h(z,w)^{-p_\Omega}$;
see \cite[Section~2]{EYZ}. Let $du$ be the normalized
$K$-invariant measure on $S$.

The scalar kernel is simplest with parameters $\sigma,\theta$ given by
\[
\nu=\theta-\sigma,\qquad
s=\frac{2n}{r}-\sigma-\theta,
\]
or, equivalently,
$\sigma=\frac nr-\frac{s+\nu}{2}$ and
$\theta=\frac nr-\frac{s-\nu}{2}$.
Let $f$ represent $\varphi$ and $F$ represent $P_{\nu,s}\varphi$.
Then \cite[Theorem~5.2]{Koranyi} gives $F=P_{\sigma,\theta}f$, where
\begin{equation}
\label{8}
P_{\sigma,\theta}f(z)
=h(z,z)^\sigma
 \int_S h(z,\omega)^{-\theta}
 h(\omega,z)^{-\sigma}f(\omega)\,du(\omega),
\qquad z\in\Omega.
\end{equation}
We use this formula below. For $f\in L^2(S)$, it defines a smooth
section on $\Omega$ through the scalar realization above.

Boundary regularity is also measured in this scalar realization.
For a section of $L^\Omega_\nu$, the notation
$C^m(\overline\Omega,L^\Omega_\nu)$ and
$C^{m,\varepsilon}(\overline\Omega,L^\Omega_\nu)$ means that its
scalar representative extends to a Euclidean neighborhood of
$\overline\Omega$ with the indicated regularity. For $q>0$, set
\begin{equation}\label{2}
\mathcal R^q(\overline\Omega,L^\Omega_\nu)
=
\begin{cases}
C^q(\overline\Omega,L^\Omega_\nu),&q\in\mathbb N,\\[1mm]
\displaystyle\bigcup_{q-m<\varepsilon\le1}
C^{m,\varepsilon}(\overline\Omega,L^\Omega_\nu),
&m=\lfloor q\rfloor<q.
\end{cases}
\end{equation}
Lemmas~\ref{28} and~\ref{29} show that these classes are preserved
by the $\widetilde G$-action.

\subsection{Jordan structure and \texorpdfstring{$K$}{K}-types}

The Jordan structure of $V$ governs the generic norm in \eqref{8},
while the $K$-action organizes the polynomial and boundary
functions used to study the transform. In the Harish--Chandra realization
$\Omega\subset V$, let $Q$ be the quadratic map from
\cite[Section~2]{KZ}. Set
\[
D(z,\bar v)w=(Q(z+w)-Q(z)-Q(w))\bar v,
\qquad z,v,w\in V.
\]
The Jordan triple product is $\{z,\bar v,w\}=D(z,\bar v)w$.
A tripotent $e\in V$ satisfies $\{e,\bar e,e\}=2e$.
Tripotents $e_1,e_2$ are orthogonal if
$\{e_1,\bar e_2,e_2\}=0$. A nonzero tripotent is minimal
if it is not a sum of two nonzero orthogonal tripotents.
Every tripotent is a sum of orthogonal minimal tripotents;
the number of summands is its rank. Thus minimal and maximal
tripotents have ranks $1$ and $r$, respectively.
For a tripotent $c$, let
\[
V_j(c)=\{v\in V:D(c,\bar c)v=jv\},\qquad j=0,1,2
\]
be its Peirce spaces. Then $V=V_2(c)\oplus V_1(c)\oplus V_0(c)$ is its Peirce
decomposition.

A Jordan frame is a family $e_1,\ldots,e_r$ of mutually
orthogonal minimal tripotents. Fix such a frame. Its joint
Peirce spaces are, for $0\le i\le j\le r$,
\begin{equation}\label{502}
V_{ij}=\bigl\{v\in V:
D(e_\ell,\bar e_\ell)v=(\delta_{i\ell}+\delta_{j\ell})v
\text{ for every }1\le\ell\le r\bigr\},
\end{equation}
where $\delta_{0\ell}=0$. We have
$V=\bigoplus_{0\le i\le j\le r}V_{ij}$, with $V_{00}=0$
and $V_{ii}=\mathbb Ce_i$; see
\cite[formulas~(20)--(21)]{EYZ}. If
$c=e_1+\cdots+e_m$ for $0\le m\le r$, its Peirce spaces satisfy
\begin{equation}\label{503}
\begin{aligned}
V_2(c)=\bigoplus_{1\le i\le j\le m}V_{ij},\ 
V_1(c)=\bigoplus_{i=1}^{m}V_{0i}
        \oplus\bigoplus_{1\le i\le m<j\le r}V_{ij},\ 
V_0(c)=\bigoplus_{j=m+1}^{r}V_{0j}
        \oplus\bigoplus_{m<i\le j\le r}V_{ij}.
\end{aligned}
\end{equation}
These identities follow from \cite[formulas~(19)--(21)]{EYZ}.

The generic norm has a simple formula in spectral coordinates.
If $z=\sum_{j=1}^r t_je_j$ is a spectral decomposition,
where $e_1,\ldots,e_r$ is a Jordan frame and zero coefficients
are allowed, then \cite[Section~2, after formula~(23)]{EYZ} gives
\begin{equation}
\label{9}h(z,z)=\prod_{j=1}^{r}(1-|t_j|^2).
\end{equation}

The group $K$ fixes the origin, acts linearly on $V$, and preserves
the Jordan structure. Write $(k\cdot F)(z)=F(k^{-1}z)$ for functions
on $\Omega$. Under this action, the holomorphic polynomial algebra has the
multiplicity-free decomposition
\cite[Section~2]{EYZ}:
\begin{equation}
\label{10}
\operatorname{Pol}(V)=\bigoplus_{\alpha}\mathcal P_{\alpha},
\end{equation}
where $\alpha=(\alpha_1,\ldots,\alpha_r)$ ranges over the
partitions $\alpha_1\ge\cdots\ge\alpha_r\ge0$.

To use these types for boundary data, fix $e_0\in S$, let
$K_1=\{k\in K:ke_0=e_0\}$. Then $S=K/K_1$. On scalar
functions on $S$, we use $(k\cdot f)(\omega)=f(k^{-1}\omega)$.
All $K$-types below refer to these actions. Let $\tau$ be an
irreducible unitary representation of $K$ on the
finite-dimensional Hilbert space $\mathcal H_\tau$, and let
$L^2(S)_\tau$ be the sum of its copies in $L^2(S)$.
With normalized Haar measure $dk$ on $K$, Schur orthogonality
gives the orthogonal projection from $L^2(S)$ to $L^2(S)_\tau$:
\[
\Pi_\tau^S f(\omega)
=\dim\mathcal H_\tau\int_K
\overline{\operatorname{tr}\tau(k)}f(k^{-1}\omega)\,dk,
\qquad \omega\in S,
\]
where $\operatorname{tr}\tau(k)$ is the trace of $\tau(k)$.
For functions on $\Omega$, the corresponding projection is
\[
\Pi_\tau^\Omega F(z)
=\dim\mathcal H_\tau\int_K
\overline{\operatorname{tr}\tau(k)}F(k^{-1}z)\,dk,
\qquad z\in\Omega.
\]
The generic norm is $K$-invariant
\cite[Section~2]{KZ}. Thus the integral in \eqref{8}
satisfies
\[
P_{\sigma,\theta}(k\cdot f)=k\cdot P_{\sigma,\theta}f,
\qquad k\in K.
\]
Averaging this identity with the character of $\tau$ yields
\begin{equation}
\label{11}
\Pi_\tau^\Omega(P_{\sigma,\theta}f)
=P_{\sigma,\theta}(\Pi_\tau^S f),
\qquad f\in L^2(S).
\end{equation}

For the polynomial types in \eqref{10}, the boundary projection
has a concrete range. Let $\mathcal P_\alpha^{K_1}$ be the
$K_1$-fixed subspace of $\mathcal P_\alpha$. By
\cite[Section~2]{EYZ},
\begin{equation}
\label{12}
\dim\mathcal P_\alpha^{K_1}=1.
\end{equation}
Restriction to $S$ is a nonzero $K$-equivariant map from
$\mathcal P_\alpha$ to $L^2(S)$. Frobenius reciprocity and
\eqref{12} show that $\mathcal P_\alpha|_S$ is the unique
copy of its $K$-type in $L^2(S)$. Thus, when $\tau$ is the
type of $\mathcal P_\alpha$, the projection $\Pi_\tau^S$
in \eqref{11} is the orthogonal projection onto
$\mathcal P_\alpha|_S$.

\subsection{Fischer inner product and expansion at the origin}

We relate the polynomial types in \eqref{10} to their restrictions
to $S$ through the Fischer inner product. First, normalize the
$K$-invariant Hermitian inner product on $V$ by
$\langle z,w\rangle=p_\Omega^{-1}\operatorname{tr}D(z,\bar w)$,
as in \cite[formula~(2)]{KZ}. For every minimal tripotent $c$,
$\operatorname{tr}D(c,\bar c)=p_\Omega$, so $\|c\|=1$.

Choose orthonormal coordinates $z=(z_1,\ldots,z_n)$ on $V$.
For $p(z)=\sum_\gamma p_\gamma z^\gamma$, set
\[
p^*(z)=\sum_\gamma\overline{p_\gamma}z^\gamma,
\qquad
p^*(\partial)=\sum_\gamma\overline{p_\gamma}\,\partial^\gamma.
\]
Here $\partial^\gamma=\partial_1^{\gamma_1}\cdots
\partial_n^{\gamma_n}$.
We use the Fischer inner product
\cite[formula~(30)]{EYZ}
\[
\langle p,q\rangle_F=(q^*(\partial)p)(0),
\qquad p,q\in\operatorname{Pol}(V).
\]
The monomials are orthogonal, with $\|z^\gamma\|_F^2=\gamma!$,
and multiplication by $p$ has Fischer adjoint $p^*(\partial)$.
For $\xi\in V$, write $\ell_\xi(z)=\langle z,\xi\rangle$ and
$\partial_\xi=\sum_{j=1}^n\xi_j\partial_j$; then
$\ell_\xi^*(\partial)=\partial_\xi$.

To state the kernel expansion, for $u\in\mathbb C$ and
$\alpha=(\alpha_1,\ldots,\alpha_r)\in\mathbb N_0^r$, put
\[
(u)_\alpha
=\prod_{\ell=1}^r
\left(u-\frac a2(\ell-1)\right)_{\alpha_\ell},
\]
where $(u)_0=1$ and $(u)_j=u(u+1)\cdots(u+j-1)$ for $j\ge1$.
For each $\alpha$, choose a Fischer-orthonormal basis
$(p_{\alpha,i})_i$ of $\mathcal P_\alpha$ and write
$K_\alpha^F(z,w)=\sum_i p_{\alpha,i}(z)
\overline{p_{\alpha,i}(w)}$ for its Fischer reproducing kernel.
The Faraut--Kor\'anyi expansion is
\begin{equation}
\label{13}
h(z,w)^{-u}=\sum_\alpha(u)_\alpha K_\alpha^F(z,w),
\qquad z,w\in\Omega,
\end{equation}
for every $u\in\mathbb C$; see \cite[formula~(33)]{EYZ}.

We also need the Fischer kernel of all homogeneous polynomials
of degree $q$. By \eqref{10}, their space is
$\bigoplus_{|\alpha|=q}\mathcal P_\alpha$; computing its kernel
from the monomial basis gives
\begin{equation}
\label{14}
\sum_{|\alpha|=q}K_\alpha^F(z,w)
=\sum_{|\gamma|=q}\frac{z^\gamma\overline{w^\gamma}}{\gamma!}
=\frac{\langle z,w\rangle^q}{q!}.
\end{equation}

Now pass from Fischer polynomials to their restrictions on $S$.
Let $H^2(S)=\overline{\operatorname{Pol}(V)|_S}^{\,L^2(S)}$
be the Hardy space. The spaces $\mathcal P_\alpha|_S$ are
mutually orthogonal, and their $L^2(S)$ inner product is
related to the Fischer inner product by
\cite[Section~2, formula~(2.8)]{UpmeierToeplitz21}:
\begin{equation}
\label{15}
\langle p_1,p_2\rangle_{S}
=\frac{1}{(n/r)_\alpha}\langle p_1,p_2\rangle_F,
\qquad p_1,p_2\in\mathcal P_\alpha.
\end{equation}
At $u=n/r$, the coefficients in \eqref{13} are therefore
exactly those needed to turn the Fischer kernels into Hardy
kernels. Thus the Cauchy--Szeg\H{o} kernel and its reproducing
formula are
\begin{equation}\label{151}
C(z,w)=h(z,w)^{-n/r},
\qquad
f(z)=\int_S C(z,\omega)f(\omega)\,du(\omega),
\quad f\in H^2(S).
\end{equation}
Here $f(z)$ denotes the holomorphic extension of $f$ to $\Omega$.

Write $(q)=(q,0,\ldots,0)$ for $q\in\mathbb N_0$.
At a rank-one element of $V$, that is, a nonzero multiple of a
minimal tripotent, the Fischer kernel of $\mathcal P_{(q)}$
has a simple form.

\begin{lemma}
\label{16}
Let $q\in\mathbb N$ and let $\xi\in V$ have rank one. Then
\begin{equation}
\label{17}
K_{(q)}^F(z,\xi)=\frac{\langle z,\xi\rangle^q}{q!}.
\end{equation}
As $\xi$ ranges over the rank-one elements of $V$, the
polynomials $\ell_\xi^q$ span $\mathcal P_{(q)}$, and the
operators $\partial_\xi^q$ span
$\{p^*(\partial):p\in\mathcal P_{(q)}\}$.
\end{lemma}

\begin{proof}
Upmeier's description of conical polynomials
\cite[Section~4, formula~(4.2) and the following paragraph]{UpmeierKinv}
implies that $\mathcal P_{(q)}$ is spanned by
$\ell_c^q$, where $c$ ranges over minimal tripotents.
Every rank-one $\xi$ is a nonzero scalar multiple of such a $c$,
so the polynomials $\ell_\xi^q$ also span $\mathcal P_{(q)}$.

By \eqref{14}, $\ell_\xi^q/q!$ represents evaluation at
$\xi$ on the full space of homogeneous polynomials of degree
$q$. It also lies in $\mathcal P_{(q)}$, so it represents
evaluation on that subspace. This gives \eqref{17}.
Taking Fischer adjoints gives the statement about
$\partial_\xi^q$.
\end{proof}

We next apply the Fischer expansion \eqref{13} to the Poisson
kernel in \eqref{8}. Since $h(0,\omega)=h(\omega,0)=1$ and
$S$ is compact, the complex powers of these factors are
defined and analytic for $z$ in a common neighborhood of $0$,
uniformly for $\omega\in S$.

For $0<t<1$, we have $t\omega\in\Omega$, so \eqref{13} gives
\[
h(z,t\omega)^{-\theta}
=\sum_\alpha(\theta)_\alpha t^{|\alpha|}
K_\alpha^F(z,\omega).
\]
Cauchy estimates on a smaller neighborhood of $0$ allow
$t\uparrow1$ and give the normally convergent expansions
\[
\begin{aligned}
h(z,\omega)^{-\theta}
&=\sum_\alpha(\theta)_\alpha K_\alpha^F(z,\omega),\\
h(\omega,z)^{-\sigma}
&=\sum_\beta(\sigma)_\beta K_\beta^F(\omega,z),
\end{aligned}
\qquad \omega\in S.
\]
The convergence is uniform in $\omega$ and persists after any
fixed number of derivatives in $z$ and $\bar z$ on a smaller
neighborhood.
Thus, for $f\in L^2(S)\subset L^1(S)$, we can multiply the
series and integrate term by term in \eqref{8}. Near $z=0$ this
gives
\begin{equation}
\label{20}
\begin{split}
h(z,z)^{-\sigma}P_{\sigma,\theta}f(z)
={}&\sum_{\alpha,\beta}(\theta)_\alpha(\sigma)_\beta
\sum_{i,i'}p_{\alpha,i}(z)\overline{p_{\beta,i'}(z)}\times
\int_S f(\omega)\overline{p_{\alpha,i}(\omega)}
p_{\beta,i'}(\omega)\,du(\omega).
\end{split}
\end{equation}
The series and its derivatives of any fixed order converge
locally uniformly near $0$.

We will later need the Hardy projection of $p\bar z_j$ for
$p\in\mathcal P_{(q)}$. In the normalization of \eqref{15},
it has the following form; see also
\cite[Lemma~15]{UpmeierToeplitz21}.

\begin{lemma}\label{21}
Let $\Pi_H:L^2(S)\to H^2(S)$ be the orthogonal
projection. For $q\in\mathbb N$ and $p\in\mathcal P_{(q)}$,
\begin{equation}\label{22}
\Pi_H(p\bar z_j|_S)
=
\frac{\partial_jp}{n/r+q-1}\Big|_S,
\qquad
1\le j\le n.
\end{equation}
\end{lemma}

\begin{proof}
By Lemma~\ref{16}, $p$ is a linear combination of $\ell_\xi^q$
with $\xi$ of rank one. Since
$\partial_j\ell_\xi^q=q\overline{\xi_j}\ell_\xi^{q-1}$,
the same spanning statement gives
$\partial_jp\in\mathcal P_{(q-1)}$.
For any homogeneous holomorphic polynomial $u$ of degree $q-1$,
Fischer adjointness, orthogonality of the polynomial $K$-types,
and \eqref{15} give
\[
\langle z_ju,p\rangle_S
=\frac{\langle z_ju,p\rangle_F}{(n/r)_q}
=\frac{\langle u,\partial_jp\rangle_F}{(n/r)_q}
=\frac{\langle u,\partial_jp\rangle_S}{n/r+q-1}.
\]
For $u$ of any other homogeneous degree, both sides vanish.
Since polynomials are dense in $H^2(S)$, this identifies the
adjoint of multiplication by $z_j$ at $p$ with
$\partial_jp/(n/r+q-1)$. The adjoint is also
$\Pi_H(p\bar z_j|_S)$, proving \eqref{22}.
\end{proof}

\subsection{Covariance and boundary regularity}

Formula~\eqref{20} gives the expansion of a Poisson transform
near the origin. To carry conclusions obtained there to other
points of $\Omega$, we use the action of $\widetilde G$ on
boundary data and the covariance of the Poisson transform.
The formulas for moving boundary data by $\widetilde G$ use
Jacobians at points of $S$. We first record why these values, and
their complex powers, are defined.

\begin{lemma}
\label{28}
Let $\widetilde g\in\widetilde G$ and put $g=\pi(\widetilde g)$.
The map $g$ and its complex Jacobian extend holomorphically to
a neighborhood of $\overline\Omega$, where the Jacobian does
not vanish. Every complex power specified by $\widetilde g$
extends to a possibly smaller neighborhood. Near the identity
in $\widetilde G$, the extensions can be chosen on a common
neighborhood and depend smoothly on $\widetilde g$.
\end{lemma}

\begin{proof}
In the Harish--Chandra realization, $\Omega$ lies in an affine
chart of its compact dual. The action of $g$ extends to an
automorphism of the compact dual
\cite[Chapter~VIII, Corollary~7.15]{Helgason}.
Because $g(\overline\Omega)=\overline\Omega$, this action and
its inverse are holomorphic near $\overline\Omega$ in the same
affine chart. Thus the Jacobian is nonzero there after the
neighborhood is made smaller.

The closed domain $\overline\Omega$ is convex, so we can choose
a convex neighborhood within the region where the Jacobian is
nonzero. Any two holomorphic logarithms of the Jacobian on $\Omega$
differ by a constant in $2\pi i\mathbb Z$. We choose the
constant in the logarithm on this neighborhood so that its
restriction to $\Omega$ is the branch specified by
$\widetilde g$. Exponentiation then extends the chosen
complex powers.

The compact-dual action depends smoothly on the group element.
Compactness of $\overline\Omega$ gives a common neighborhood for
$\widetilde g$ near the identity. After shrinking it, the
Jacobians are uniformly close to $1$, so their logarithms and
complex powers also depend smoothly on $\widetilde g$.
\end{proof}

For $\widetilde g_1,\widetilde g_2\in\widetilde G$, with
projections $g_1,g_2$, the chain rule and the chosen branches give
\begin{equation}\label{23}
J_{\widetilde g_1\widetilde g_2}(z)^\kappa
=J_{\widetilde g_1}(g_2z)^\kappa
 J_{\widetilde g_2}(z)^\kappa,
\qquad z\in\overline\Omega,\quad \kappa\in\mathbb C.
\end{equation}
In particular,
$J_{\widetilde g^{-1}}(gz)^\kappa
=J_{\widetilde g}(z)^{-\kappa}$.
The generic norm transforms by
\begin{equation}\label{24}
h(gz,gw)
=J_{\widetilde g}(z)^{1/p_\Omega}h(z,w)
 \overline{J_{\widetilde g}(w)}^{\,1/p_\Omega},
\qquad z\in\Omega,\quad w\in\overline\Omega,
\end{equation}
for $g=\pi(\widetilde g)$; see
\cite[formulas~(3.21) and~(3.23)]{Koranyi}.
The normalized $K$-invariant measure on $S$ satisfies
\begin{equation}\label{25}
du(g\omega)
=|J_{\widetilde g}(\omega)|^{2n/(rp_\Omega)}\,du(\omega),
\qquad \omega\in S;
\end{equation}
see \cite[Section~1]{MisraUpmeierBoundary}.

The $\widetilde G$-action on $L^S_{\nu;s}$ induces an action
on its scalar representatives. With
$\nu=\theta-\sigma$ and $s=2n/r-\sigma-\theta$, it is
\begin{equation}\label{26}
(\pi_{\sigma,\theta}(\widetilde g)f)(\omega)
=J_{\widetilde g^{-1}}(\omega)^{(n/r-\sigma)/p_\Omega}
 \overline{J_{\widetilde g^{-1}}(\omega)}^{(n/r-\theta)/p_\Omega}
 f(g^{-1}\omega).
\end{equation}
The cocycle identity \eqref{23} shows that
$\pi_{\sigma,\theta}(\widetilde g^{-1})$ is its inverse.
Both operators are bounded on $L^2(S)$ by Lemma~\ref{28}
and the change-of-measure formula \eqref{25}.

Changing variables in \eqref{8} and using \eqref{24}--\eqref{26},
we obtain
\begin{equation}\label{27}
\bigl(P_{\sigma,\theta}(\pi_{\sigma,\theta}(\widetilde g)f)\bigr)(z)
=J_{\widetilde g^{-1}}(z)^{(\theta-\sigma)/p_\Omega}
 (P_{\sigma,\theta}f)(g^{-1}z).
\end{equation}
The right-hand side is the action $T_{\theta-\sigma}(\widetilde g)$
from \eqref{6} on $P_{\sigma,\theta}f$. We will use \eqref{27}
to apply calculations at the origin to transforms at other
points of $\Omega$. The next lemma shows that this action and
the $K$-type projections in \eqref{11} preserve the required
boundary regularity.

\begin{lemma}
\label{29}
Let $\sigma,\theta\in\mathbb C$, $f\in L^2(S)$, and
$F=P_{\sigma,\theta}f$. If $F\in C^m(\overline\Omega)$,
where $m\in\mathbb N_0$, then
$P_{\sigma,\theta}\pi_{\sigma,\theta}(\widetilde g)f$
and $P_{\sigma,\theta}\Pi_\tau^S f$ also belong to
$C^m(\overline\Omega)$ for every $\widetilde g\in\widetilde G$
and every $K$-type $\tau$. The same holds for
$C^{m,\varepsilon}$, $0<\varepsilon\le1$, and for
$\mathcal R^q$, $q>0$.
\end{lemma}

\begin{proof}
Equation~\eqref{27} expresses the first transform as a
holomorphic, nonvanishing Jacobian factor times $F\circ g^{-1}$.
By Lemma~\ref{28}, the map $g^{-1}$ and the Jacobian
factor extend smoothly to a neighborhood of
$\overline\Omega$. Composition and multiplication
therefore preserve the stated $C^m$ and H\"older regularity.

For the projection, \eqref{11} gives
$P_{\sigma,\theta}\Pi_\tau^S f=\Pi_\tau^\Omega F$.
Choose a $K$-invariant neighborhood of $\overline\Omega$ on
which $F$ has an extension. Since $K$ is compact and acts
linearly, the functions $F\circ k^{-1}$ have uniform $C^m$
and, when applicable, $C^{m,\varepsilon}$ bounds on a smaller
such neighborhood. Integrating them against the character in
the definition of $\Pi_\tau^\Omega$ preserves these bounds.
The claim for $\mathcal R^q$ follows from its definition.
\end{proof}

\subsection{Covariant differential operators}

We use the spaces $\mathcal P_{(q)}$ to define the
differential operators in Theorem~\ref{3}. For
$q\in\mathbb N$, set
\[
\mathcal W_q=\{p^*:p\in\mathcal P_{(q)}\},
\qquad
\mathcal W_q^\vee=\operatorname{Hom}_{\mathbb C}
(\mathcal W_q,\mathbb C).
\]
The elements of $\mathcal W_q$ are polynomial symbols for
constant-coefficient differential operators.

\begin{definition} For
$F\in C^q(\Omega)$, define the $\mathcal W_q^\vee$-valued
function $\mathcal D_qF$ by
\[
(\mathcal D_qF(z))(p^*)=p^*(\partial)F(z),
\qquad p\in\mathcal P_{(q)}.
\]
\end{definition}

For
$L\in\mathbb N_0$, set
$\mathcal F_L=\bigoplus_{\alpha_1\le L}\mathcal P_\alpha$.
By \cite[Theorem~3.5]{UpmeierKinv}, the sum of the
remaining types in $ \operatorname{Pol}(V)$ is the polynomial ideal
\begin{equation}
\label{30}
\mathcal P_{(L+1)}\operatorname{Pol}(V)
:=\operatorname{span}\{pR:p\in\mathcal P_{(L+1)},
R\in\operatorname{Pol}(V)\}
=\bigoplus_{\alpha_1\ge L+1}\mathcal P_\alpha.
\end{equation}

\begin{lemma}\label{32}
For $L\in\mathbb N_0$ and $F\in\operatorname{Pol}(V)$,
$\mathcal D_{L+1}F$ vanishes on $\Omega$ if and only if
$F\in\mathcal F_L$.
\end{lemma}

\begin{proof}
By definition, $\mathcal D_{L+1}F\equiv0$ if and only if
$p^*(\partial)F=0$ for every $p\in\mathcal P_{(L+1)}$.
For such $p$ and any $R\in\operatorname{Pol}(V)$, Fischer
adjointness gives
\[
\langle pR,F\rangle_F
=\langle R,p^*(\partial)F\rangle_F.
\]
Consequently, $\mathcal D_{L+1}F\equiv0$ implies
$F\perp\mathcal P_{(L+1)}\operatorname{Pol}(V)$.
Conversely, suppose that
$F\perp\mathcal P_{(L+1)}\operatorname{Pol}(V)$.
For each $p\in\mathcal P_{(L+1)}$, the function
$p^*(\partial)F$ is a polynomial. Taking
$R=p^*(\partial)F$ in the displayed identity gives
\[
\|p^*(\partial)F\|_F^2
=\langle p\,p^*(\partial)F,F\rangle_F=0.
\]
Thus $p^*(\partial)F=0$ for every such $p$, which means
$\mathcal D_{L+1}F\equiv0$.
By \eqref{30}, the orthogonal complement of
$\mathcal P_{(L+1)}\operatorname{Pol}(V)$ is
$\bigoplus_{\alpha_1\le L}\mathcal P_\alpha=\mathcal F_L$. The proof is complete.
\end{proof}

For the transformation law below, we also need to express
$\mathcal F_L$ in terms of the generic norm.

\begin{lemma}
For $L\in\mathbb N_0$, the space $\mathcal F_L$ also has
the description
\begin{equation}
\label{31}
\mathcal F_L
=\operatorname{span}\{h(\,\cdot\,,w)^L:w\in\Omega\}.
\end{equation}
\end{lemma}

\begin{proof}
Denote the span on the right of \eqref{31} by $W_L$.
Because $\alpha$ is a partition,
$(-L)_\alpha\ne0$ exactly when $\alpha_1\le L$.
Formula~\eqref{13} gives the finite expansion
\[
h(\,\cdot\,,w)^L
=\sum_{\alpha_1\le L}(-L)_\alpha
K_\alpha^F(\,\cdot\,,w),
\qquad w\in\Omega.
\]
Thus $W_L\subset\mathcal F_L$. Suppose that
$p=\sum_{\alpha_1\le L}p_\alpha\in\mathcal F_L$
is Fischer-orthogonal to $W_L$. By the reproducing
property of $K_\alpha^F$, for every $w\in\Omega$,
\[
0=\langle p,h(\,\cdot\,,w)^L\rangle_F
 =\sum_{\alpha_1\le L}(-L)_\alpha p_\alpha(w).
\]
The polynomial on the right vanishes on $\Omega$, hence
identically equals to $0$ on $V$. Its terms belong to distinct spaces
$\mathcal P_\alpha$, so each $(-L)_\alpha p_\alpha=0$.
All these coefficients are nonzero; thus $p=0$.
Since $\mathcal F_L$ is finite-dimensional,
$W_L=\mathcal F_L$, proving \eqref{31}.
\end{proof}

To describe how $\mathcal D_q$ changes under automorphisms,
we first check that their derivatives preserve its symbol
space $\mathcal W_q$. By Lemma~\ref{16}, this space is
spanned by $(\ell_\xi^q)^*$ with $\xi$ of rank one. For a
linear map $A:V\to V$,
\[
(\ell_\xi^q)^*(A^t\eta)=(\ell_{A\xi}^q)^*(\eta).
\]Hence, if both $A$ and $A^{-1}$ preserve rank-one elements,
substitution by $A^t$ is an automorphism of $\mathcal W_q$.
For $g\in G$, we will apply this with $A=(g^{-1})'(z)$, so we need to
verify this rank-preserving property for derivatives of
automorphisms. Kaup shows that $g'(z)$ belongs
to the structure group of $V$ for $g\in G$ and $z\in\Omega$
\cite[p.~249]{Kaup02}, and that the orbits of this group are
the sets of elements of each Jordan rank
\cite[p.~253]{Kaup02}. Applying this to $g^{-1}$ shows that
$(g^{-1})'(z)$ and its inverse preserve rank one. We can
therefore define an automorphism of $\mathcal W_q$ by
\[
(C_g(z)\varphi)(\eta)
=
\varphi\!\left([(g^{-1})'(z)]^t\eta\right),
\qquad \varphi\in\mathcal W_q.
\]
Its dual map $B_g(z)$ acts on $\mathcal W_q^\vee$ by
\begin{equation}\label{33}
(B_g(z)\lambda)(\varphi)=\lambda(C_g(z)\varphi),
\qquad \lambda\in\mathcal W_q^\vee,
\quad \varphi\in\mathcal W_q.
\end{equation}
Both maps depend holomorphically on $z$, since
$(g^{-1})'(z)$ is holomorphic and $\mathcal W_q$ is
finite-dimensional. The chain rule gives
\[
B_{g_1g_2}(z)=B_{g_1}(z)B_{g_2}(g_1^{-1}z),
\qquad g_1,g_2\in G.
\]
Consequently, the product holomorphic bundle
$\mathcal V_q=\Omega\times\mathcal W_q^\vee$ has a $G$-action
\[
g\cdot(z,\lambda)=(gz,B_g(gz)\lambda).
\]
We use this bundle as the target of the differential operators
below.

The spanning identity \eqref{31} makes the following
transformation law for $\mathcal D_{L+1}$ possible.

\begin{proposition}
Let $L\in\mathbb N_0$. For $\widetilde g\in\widetilde G$, put
$g=\pi(\widetilde g)$.  Then, for every $F\in C^{L+1}(\Omega)$,
the following identity holds in $\mathcal W_{L+1}^\vee$ at each $z\in\Omega$:
\begin{equation}\label{34}
\mathcal D_{L+1}\!\left[
J_{\widetilde g^{-1}}^{-L/p_\Omega}
(F\circ g^{-1})
\right](z)
=
J_{\widetilde g^{-1}}(z)^{-L/p_\Omega}
B_g(z)
(\mathcal D_{L+1}F)(g^{-1}z).
\end{equation}
\end{proposition}

\begin{proof}
Fix $\widetilde g\in\widetilde G$ and $z\in\Omega$, and put
\(y=g^{-1}z\) and \(A=(g^{-1})'(z)\).
By Lemma~\ref{16}, $\mathcal W_{L+1}$ is spanned by
$\varphi_\xi=(\ell_\xi^{L+1})^*$ with $\xi$ of rank one.
The definition of $\mathcal D_{L+1}$ gives
\[
\bigl(\mathcal D_{L+1}F(y)\bigr)(\varphi_\xi)
=\partial_\xi^{L+1}F(y).
\]
The product and chain rules give
\[
\partial_\xi^{L+1}\!\left[
J_{\widetilde g^{-1}}^{-L/p_\Omega}(F\circ g^{-1})
\right](z)
=J_{\widetilde g^{-1}}(z)^{-L/p_\Omega}
\partial_{A\xi}^{L+1}F(y)+R_{\xi,z}(F),
\]
where $R_{\xi,z}(F)$ depends only on derivatives of $F$ at $y$
of order at most $L$. Since
$C_g(z)\varphi_\xi=\varphi_{A\xi}$, the first term on the right is
\[
J_{\widetilde g^{-1}}(z)^{-L/p_\Omega}
\bigl(B_g(z)\mathcal D_{L+1}F(y)\bigr)(\varphi_\xi).
\]
The $\varphi_\xi$ span $\mathcal W_{L+1}$, so the difference
in \eqref{34} involves only $\partial^\gamma F(y)$,
$|\gamma|\le L$.

The space $\mathcal F_L$ is invariant under \(F\longmapsto J_{\widetilde g^{-1}}^{-L/p_\Omega} (F\circ g^{-1})\).
Indeed, by \eqref{24},
\[
J_{\widetilde g^{-1}}(z)^{-L/p_\Omega}
h(g^{-1}z,w)^L
=
\overline{
J_{\widetilde g^{-1}}(gw)
}^{\,L/p_\Omega}
h(z,gw)^L,
\]
and the assertion follows from \eqref{31}.

For $P\in\mathcal F_L$, the invariance just proved shows
that its weighted pullback also belongs to $\mathcal F_L$.
Lemma~\ref{32} therefore makes both sides of \eqref{34}
zero. Hence $R_{\xi,z}(P)=0$ for every $P\in\mathcal F_L$.

For a general $F$, define
\[
P_y(w)=\sum_{|\gamma|\le L}
\frac{\partial^\gamma F(y)}{\gamma!}(w-y)^\gamma.
\]
This polynomial has total degree at most $L$, so
$P_y\in\mathcal F_L$. By construction,
\[
\partial^\gamma P_y(y)=\partial^\gamma F(y),
\qquad |\gamma|\le L.
\]
Since $R_{\xi,z}$ depends only on these derivatives,
\[
R_{\xi,z}(F)=R_{\xi,z}(P_y)=0.
\]
Thus \eqref{34} holds after evaluation on every
$\varphi_\xi$. They span $\mathcal W_{L+1}$,
so \eqref{34} holds in $\mathcal W_{L+1}^\vee$.
\end{proof}

The covariance in \eqref{34} lets us check a differential
equation at the origin after moving the transform by
$\widetilde G$.
Since $\mathcal W_{L+1}^\vee$ is a fixed vector space,
we apply $\bar\partial$ componentwise to
$\mathcal D_{L+1}F$. The result is a
$\mathcal W_{L+1}^\vee$-valued $(0,1)$-form.

\begin{lemma}
\label{35}
Let $L=\sigma-\frac nr\in\mathbb N_0$ and
$F=P_{\sigma,\frac nr}f$ for $f\in L^2(S)$.
For $\widetilde g\in\widetilde G$, put
$F_{\widetilde g}=P_{\sigma,\frac nr}
(\pi_{\sigma,\frac nr}(\widetilde g)f)$.
If $\bigl(\bar\partial(\mathcal D_{L+1}F_{\widetilde g})\bigr)_0=0$
for every $\widetilde g$, then
$\bar\partial(\mathcal D_{L+1}F)=0$ on $\Omega$.
\end{lemma}

\begin{proof}
Write $g=\pi(\widetilde g)$. By \eqref{27} and
$L=\sigma-\frac nr$,
\[
F_{\widetilde g}(z)
=
J_{\widetilde g^{-1}}(z)^{-L/p_\Omega}
F(g^{-1}z).
\]
Equation~\eqref{34} therefore gives
\[
\mathcal D_{L+1}F_{\widetilde g}(z)
=
J_{\widetilde g^{-1}}(z)^{-L/p_\Omega}
B_g(z)
\bigl(\mathcal D_{L+1}F\bigr)(g^{-1}z).
\]
Differentiating under the integral in \eqref{8} shows
that the transforms are smooth on $\Omega$. The Jacobian
factor and $B_g(z)$ are holomorphic in $z$.
Taking an antiholomorphic derivative at $0$ and using the
hypothesis gives
\[
 0
=
J_{\widetilde g^{-1}}(0)^{-L/p_\Omega}
B_g(0)
\circ
\bigl(\bar\partial(\mathcal D_{L+1}F)\bigr)_{g^{-1}0}
\circ
\overline{(g^{-1})'(0)}.
\]
The Jacobian factor is nonzero, and both
$B_g(0)$ and $(g^{-1})'(0)$ are invertible. Hence
$\bigl(\bar\partial(\mathcal D_{L+1}F)\bigr)_{g^{-1}0}=0$.
Since $G$ acts transitively on $\Omega$, the result follows.
\end{proof}

Equation~\eqref{34} uses the Jacobian exponent $-L/p_\Omega$.
To obtain the exponents $-(L+j_0)/p_\Omega$ needed in
Theorem~\ref{3}, we insert a power of the generic norm.

\begin{definition}If $L,j_0\in\mathbb N_0$ and $F$ is the scalar coefficient
of a $C^{L+1}$ section of $L^\Omega_{-(L+j_0)}$, define
\begin{equation}\label{36}
\mathbb D_{L,j_0}F
=
h(z,z)^{j_0}
\mathcal D_{L+1}\bigl(h(z,z)^{-j_0}F\bigr).
\end{equation}
\end{definition}

\begin{proposition}\label{37}
Let $L,j_0\in\mathbb N_0$, and let $F$ be the scalar representative
of a $C^{L+1}$ section of
$L^\Omega_{-(L+j_0)}$. For $\widetilde g\in\widetilde G$,
put $g=\pi(\widetilde g)$ and
$F_{\widetilde g}=T_{-(L+j_0)}(\widetilde g)F$, as in
\eqref{6}. Then, as an identity in $\mathcal W_{L+1}^\vee$,
\begin{equation}\label{38}
\mathbb D_{L,j_0}F_{\widetilde g}(z)
=
J_{\widetilde g^{-1}}(z)^{-(L+j_0)/p_\Omega}
B_g(z)
(\mathbb D_{L,j_0}F)(g^{-1}z).
\end{equation}
\end{proposition}

\begin{proof}
Put \(H(z)=h(z,z)^{-j_0}F(z)\) and
\(y=g^{-1}z\).
The diagonal case of \eqref{24} gives
\[
h(z,z)^{-j_0}F_{\widetilde g}(z)
=
\overline{J_{\widetilde g^{-1}}(z)}^{\,j_0/p_\Omega}
J_{\widetilde g^{-1}}(z)^{-L/p_\Omega}
H(g^{-1}z).
\]
The conjugate Jacobian factor is antiholomorphic. Thus, for every
$\varphi\in\mathcal W_{L+1}$, it passes through the constant-coefficient
operator $\varphi(\partial)$ without being differentiated. Applying
\eqref{34} to $H$ therefore gives the following identity in
$\mathcal W_{L+1}^\vee$:
\begin{equation}\label{380}
\begin{aligned}
\mathcal D_{L+1}
\bigl(h(z,z)^{-j_0}F_{\widetilde g}\bigr)(z)
={}&
\overline{J_{\widetilde g^{-1}}(z)}^{\,j_0/p_\Omega}
J_{\widetilde g^{-1}}(z)^{-L/p_\Omega}
B_g(z)(\mathcal D_{L+1}H)(y).
\end{aligned}
\end{equation}
Again by the diagonal case of \eqref{24},
\begin{equation}\label{381}
h(z,z)^{j_0}
\overline{J_{\widetilde g^{-1}}(z)}^{\,j_0/p_\Omega}
=
J_{\widetilde g^{-1}}(z)^{-j_0/p_\Omega}
h(g^{-1}z,g^{-1}z)^{j_0}.
\end{equation}
Multiplying \eqref{380} by $h(z,z)^{j_0}$, then using
\eqref{381} and \eqref{36}
gives \eqref{38}.
\end{proof}

Formula~\eqref{36} uses derivatives of order at most $L+1$.
Its transformation law \eqref{38}, together with the action
on $\mathcal V_{L+1}$ defined above, shows that it is the
scalar expression of a differential operator between bundles:
\[
\mathbb D_{L,j_0}:
C^{L+1}(\Omega,L^\Omega_{-(L+j_0)})
\longrightarrow
C^0(\Omega,L^\Omega_{-(L+j_0)}\otimes\mathcal V_{L+1}).
\]

The target bundle is holomorphic. For smooth $F$, its
Dolbeault operator $\bar\partial$ can therefore act on
$\mathbb D_{L,j_0}F$; in the Harish--Chandra trivialization,
it differentiates each component as above. In particular,
$\mathbb D_{L,0}=\mathcal D_{L+1}$ in this trivialization,
so the equation in Lemma~\ref{35} is the $j_0=0$ case.

\section{Differential equations and parameter analysis}\label{39}

We first record the differential equations satisfied by the Poisson
transform. The rank-one formulas below will be used to determine
the possible parameters.

\subsection{Differential identities and rank-one formulas}
Let \begin{equation}\label{385}B(z,\bar w)=I-D(z,\bar w)+Q(z)Q(\bar w)\end{equation}
be the Bergman operator.  Formula~(22) of \cite{KZ} gives 
\begin{equation}\label{386}-\partial_i\bar\partial_j\log h(z,z) = \bigl(B(z,\bar z)^{-1}\bigr)_{ji}.\end{equation}
Write
\begin{equation}\label{387}
b^{i\bar j}(z)=\bigl(B(z,\bar z)\bigr)_{ij},
\qquad
\Delta_\Omega F
=
\sum_{i,j}
b^{i\bar j}(z)\partial_i\bar\partial_jF.
\end{equation}

\begin{lemma}\label{40}
Let $\sigma,\theta\in\mathbb C$, let $f\in L^2(S)$, and put \(F=P_{\sigma,\theta}f\).
Then
\begin{equation}\label{41}
\Delta_\Omega F
+(\sigma-\theta)\sum_j
\overline{(z-Q(z)\bar z)_j}\,\bar\partial_jF
+r\sigma\left(\frac nr-\theta\right)F
=0.
\end{equation}
\end{lemma}

\begin{proof}
Combining \cite[Theorems~6.2 and~6.3, formulas~(6.20) and~(6.22)]{Koranyi}
and using $\ell=\theta-\sigma$ and
$s=2n/r-\sigma-\theta$, we obtain, in the scalar trivialization,
\begin{equation}\label{388}
h(z,z)^{\sigma-\theta}
\sum_{i,j}b^{i\bar j}(z)
\partial_i\left(
h(z,z)^{-(\sigma-\theta)}
\bar\partial_jF
\right)
+
r\sigma\left(\frac nr-\theta\right)F
=0.
\end{equation}
Taking the complex conjugate of \cite[formula~(46)]{EYZ}
and using our notation for the Bergman operator gives
\[
\bar\partial(-\log h(z,z))
=
B(z,\bar z)^{-1}(z-Q(z)\bar z).
\]
Since $B(z,\bar z)$ is Hermitian,
\[
\sum_i
b^{i\bar j}(z)\partial_i\log h(z,z)
=
-\overline{(z-Q(z)\bar z)_j}.
\]
Expanding \eqref{388} and using the preceding identity gives
\eqref{41}.
\end{proof}

\begin{lemma}
Let $c$ be a minimal tripotent and put
$\ell_c(z)=\langle z,c\rangle$. For $0\le t<1$ and $\omega\in S$,
\begin{equation}\label{42}
h(\omega,tc)=1-t\ell_c(\omega),\qquad
h(tc,\omega)=1-t\overline{\ell_c(\omega)},\qquad
h(tc,tc)=1-t^2.
\end{equation}
\end{lemma}

\begin{proof}
For each $m\ge0$, \eqref{14} and \eqref{17}, together with
the direct-sum decomposition \eqref{10}, give
$K_\alpha^F(z,c)=0$ when $|\alpha|=m$ and $\alpha\ne(m)$.
Thus \eqref{13} with $u=-1$ gives, for $z\in\Omega$,
\[
h(z,tc)
=\sum_{m\ge0}(-1)_m t^mK_{(m)}^F(z,c)
=1-t\ell_c(z).
\]
Both sides are polynomials in $z$, so the identity holds
on $V$. Taking $z=\omega$ gives the first identity in
\eqref{42}. The second follows from
$h(tc,\omega)=\overline{h(\omega,tc)}$.
Taking $z=tc$ and using $\langle c,c\rangle=1$
gives the third.
\end{proof}

Since $\ell_c^m\in\mathcal P_{(m)}$ and
$\|\ell_c^m\|_F^2=m!$, the orthogonality of distinct
$\mathcal P_\alpha|_S$ and \eqref{15} give, for
$m,m'\in\mathbb N_0$,
\begin{equation}\label{389}
\int_S \ell_c(\omega)^m
\overline{\ell_c(\omega)}^{\,m'}\,du(\omega)
=
\delta_{m,m'}\frac{m!}{(\frac nr)_m}.
\end{equation}

These identities will be used in the following rank-one calculation.

\begin{lemma}\label{43}
Let $c$ be a minimal tripotent and let $0\le t<1$.
For $\sigma,\theta\in\mathbb C$ and $q\in\mathbb N_0$,
\begin{equation}\label{44}
P_{\sigma,\theta}(\ell_c^q)(tc)
=
\frac{(\theta)_q}{(\frac nr)_q}
t^q(1-t^2)^\sigma
{}_2F_1\left(
\sigma,\theta+q;\frac nr+q;t^2
\right).
\end{equation}
\end{lemma}

\begin{proof}
By \eqref{42},
\[
P_{\sigma,\theta}(\ell_c^q)(tc)
=
(1-t^2)^\sigma
\int_S
(1-t\overline{\ell_c})^{-\theta}
(1-t\ell_c)^{-\sigma}\ell_c^q\,du.
\]
Expanding the two factors and using \eqref{389} gives
\[
P_{\sigma,\theta}(\ell_c^q)(tc)
=
(1-t^2)^\sigma t^q
\sum_{m\ge0}
\frac{(\theta)_{m+q}(\sigma)_m}
     {(\frac nr)_{m+q}m!}\,t^{2m}.
\]
Factoring out $(\theta)_q/(\frac nr)_q$ and using the
definition of ${}_2F_1$ proves \eqref{44}.
\end{proof}

\begin{remark}
Setting $q=0$ in \eqref{44} gives
\begin{equation}\label{45}
P_{\sigma,\theta}1(tc)
=
(1-t^2)^\sigma
{}_2F_1\left(
\sigma,\theta;\frac nr;t^2
\right).
\end{equation}

The kernel in \eqref{8} also gives
\[
P_{\sigma,\theta}\overline f
=
h(z,z)^{\sigma-\theta}
\overline{P_{\bar\theta,\bar\sigma}f}.
\]
Applying \eqref{44} with $q=1$ and parameters
$(\bar\theta,\bar\sigma)$ therefore yields
\begin{equation}\label{46}
P_{\sigma,\theta}(\overline{\ell_c})(tc)
=
\frac{r\sigma}{n}\,
t(1-t^2)^\sigma
{}_2F_1\left(
\theta,\sigma+1;\frac nr+1;t^2
\right).
\end{equation}

Finally, Euler's transformation rewrites \eqref{44} as
\begin{equation}\label{47}
P_{\sigma,\theta}(\ell_c^q)(tc)
=
\frac{(\theta)_q}{(\frac nr)_q}
t^q(1-t^2)^{\frac nr-\theta}
{}_2F_1\left(
\frac nr+q-\sigma,\frac nr-\theta;
\frac nr+q;t^2
\right).
\end{equation}
\end{remark}

\subsection{Boundary regularity and parameter reduction}

\begin{lemma}\label{48}
Let $C>0$ and $A,B\in\mathbb C$.  The function \[\psi(\rho) = \rho^A{}_2F_1(A,B;C;1-\rho), \quad 0<\rho\le1,\]
extends smoothly across $\rho=0$ if and only if at least one of
\begin{equation}\label{49}
A=0;\qquad
B=C;\qquad
A\in\mathbb N_0,\ B\in-\mathbb N_0;\qquad
A-C\in\mathbb N_0,\ C-B\in\mathbb N_0
\end{equation}
holds.  In each of these cases $\psi$ is a polynomial in $\rho$.  Otherwise
$\psi$ fails to have a $C^m$ extension at $0$ for some
$m\in\mathbb N_0$.
\end{lemma}

\begin{proof}
The four alternatives in \eqref{49} are sufficient. If $A=0$ or
$B=C$, then $\psi=1$. In the third case the Gauss series
terminates. In the fourth case Euler's transformation gives
\[
\psi(\rho)=
\rho^{C-B}{}_2F_1(C-A,C-B;C;1-\rho),
\]
which is again a polynomial.

Conversely, suppose that $\psi$ is smooth at $\rho=0$.
The Gauss hypergeometric equation gives
\[
\rho(1-\rho)\psi''
+\bigl(B+1-A-C+(A-B-1)\rho\bigr)\psi'
+\frac{A(C-B)}{\rho}\psi=0.
\]
Its only possible singularities are $0$, $1$, and $\infty$,
all regular. The Frobenius expansions show that smoothness
at $0$ forces $\psi$ to be analytic there; see
\cite[Section~2.7(i)]{DLMF}.
The defining series shows that $\psi$ is also analytic at $1$.
Continuation around both finite singularities is therefore
trivial, so $\psi$ extends to an entire function.

The point at infinity is a regular singularity. Its Frobenius
expansions therefore give constants $C_0,M>0$ such that
$|\psi(\rho)|\leq C_0(1+|\rho|)^M$ for large $|\rho|$.
For any integer $n>M$, Cauchy's estimate on $|\rho|=R$ gives
\[
|\psi^{(n)}(0)|
\leq \frac{n!C_0(1+R)^M}{R^n}.
\]
Letting $R\to\infty$ shows that $\psi^{(n)}(0)=0$.
Hence $\psi$ is a polynomial.

If $\psi$ is constant, then $\psi(1)=1$ and the equation gives
$A=0$ or $B=C$. Otherwise, let $N\geq1$ be its degree and
$m\geq0$ its order of vanishing at $0$. Comparing the
highest and lowest powers of $\rho$ gives
\[
N=A-B,\qquad (m-A)(m-C+B)=0.
\]
Since $0\leq m\leq N$, either
$A=m\in\mathbb N_0$ and $B=m-N\in-\mathbb N_0$, or
$C-B=m\in\mathbb N_0$ and $A-C=N-m\in\mathbb N_0$.
These are the remaining two conditions in \eqref{49}.
\end{proof}

\begin{lemma}\label{50}
Let $\sigma,\theta\in\mathbb C$, let
$m\in\mathbb N_0\cup\{\infty\}$, and let $f\in L^2(S)$.  If \[0\ne P_{\sigma,\theta}f\in C^m(\overline\Omega),\]
then $P_{\sigma,\theta}1\in C^m(\overline\Omega)$.  If
$m\in\mathbb N_0$ and $0<\varepsilon\le1$, the same statement holds with
$C^m$ replaced by $C^{m,\varepsilon}$.  For every $s>0$, it also holds with
$C^m$ replaced by $\mathcal R^s$.
\end{lemma}

\begin{proof}
Let $F=P_{\sigma,\theta}f$ and choose $z_0\in\Omega$ with $F(z_0)\ne0$.
Choose $g\in G$ with $g^{-1}0=z_0$ and choose a lift
$\widetilde g\in\widetilde G$.  Put
\[
 f_{\widetilde g}=\pi_{\sigma,\theta}(\widetilde g)f,
 \qquad
 F_{\widetilde g}=P_{\sigma,\theta}f_{\widetilde g}.
\]
By \eqref{27}, \(F_{\widetilde g}(0) =J_{\widetilde g^{-1}}(0)^{(\theta-\sigma)/p_\Omega}F(z_0)\),
so $F_{\widetilde g}(0)\ne0$.  Lemma~\ref{29} gives the same boundary
regularity for $F_{\widetilde g}$ as for $F$. When $m=\infty$, both
uses of Lemma~\ref{29} in this proof are made at every finite order.

Let \(f_0=\int_K f_{\widetilde g}\circ k^{-1}\,dk\).
Since $K$ acts transitively on $S$, and $h(0,\omega)=h(\omega,0)=1$,
\[
f_0(w)=\int_S f_{\widetilde g}(\omega)\,du(\omega)
=F_{\widetilde g}(0),\qquad w\in S.
\]

Since $P_{\sigma,\theta}$ intertwines the $K$-actions,
\[
P_{\sigma,\theta}f_0
=\int_K F_{\widetilde g}(k^{-1}z)\,dk.
\]
As $f_0\equiv F_{\widetilde g}(0)$, we obtain
\[
F_{\widetilde g}(0)P_{\sigma,\theta}1
=\int_K F_{\widetilde g}(k^{-1}z)\,dk.
\]
The right-hand side is the $K$-average of $F_{\widetilde g}$ and therefore
has the same $C^m$, $C^{m,\varepsilon}$, or $\mathcal R^s$ boundary
regularity; see Lemma~\ref{29}. Since $F_{\widetilde g}(0)\ne0$,
$P_{\sigma,\theta}1$ has the same regularity.
\end{proof}

\begin{proposition}\label{51}
Let $\sigma,\theta\in\mathbb C$.  If \(P_{\sigma,\theta}(L^2(S)) \cap C^\infty(\overline\Omega)\ne\{0\}\),
then at least one of the following holds:
\begin{equation}\label{52}
\sigma=0;\qquad
\theta=\frac nr;\qquad
\sigma\in\mathbb N_0,\ \theta\in-\mathbb N_0;\qquad
\sigma-\frac nr\in\mathbb N_0,\ 
\frac nr-\theta\in\mathbb N_0.
\end{equation}

At every parameter outside \eqref{52}, there exists
$m\in\mathbb N_0$ such that, for every $f\in L^2(S)$,
\begin{equation}\label{53}
P_{\sigma,\theta}f\in C^m(\overline\Omega)
\quad\Longrightarrow\quad
P_{\sigma,\theta}f=0.
\end{equation}
\end{proposition}

\begin{proof}
Suppose first that $P_{\sigma,\theta}f$ is nonzero and smooth on
$\overline\Omega$.  By Lemma~\ref{50},
$P_{\sigma,\theta}1$ has the same regularity.  For a minimal tripotent $c$, \eqref{45} gives
\[
P_{\sigma,\theta}1(tc)
=
(1-t^2)^\sigma
{}_2F_1\left(
\sigma,\theta;\frac nr;t^2
\right).
\]
Putting $\rho=1-t^2$ and applying Lemma~\ref{48} with \(A=\sigma,\quad B=\theta,\quad C=\frac nr\)
gives exactly the four alternatives in \eqref{52}.

Now suppose that $(\sigma,\theta)$ lies outside
\eqref{52}.  Lemma~\ref{48} shows that the
preceding radial function fails to have a $C^m$ extension at $t=1$ for
some finite $m\in\mathbb N_0$.  If a nonzero
$P_{\sigma,\theta}f$ belonged to $C^m(\overline\Omega)$, then
Lemma~\ref{50} would give the same regularity for
$P_{\sigma,\theta}1$, a contradiction.  Hence
\eqref{53} holds.
\end{proof}

\begin{proposition}\label{54}
Let $f\in L^2(S)$. Then $P_{0,\theta}f$ is holomorphic
for every $\theta\in\mathbb C$.
For $m\in\mathbb N$ and $q\in\mathbb N_0$,
\begin{equation}\label{55}
\mathcal D_{q+1}\bigl(h(z,z)^{-m}P_{m,-q}f(z)\bigr)=0.
\end{equation}
\end{proposition}

\begin{proof}
If $\sigma=0$, then \(P_{0,\theta}f(z) = \int_S h(z,\omega)^{-\theta}f(\omega)\,du(\omega)\),
which is holomorphic in $z$.

For the second assertion, we recall that
\[
h(z,z)^{-m}P_{m,-q}f(z)
=
\int_S h(z,\omega)^q h(\omega,z)^{-m}f(\omega)\,du(\omega).
\]
By \eqref{31}, $h(\,\cdot\,,w)^q\in\mathcal F_q$
for $w\in\Omega$. Since $\mathcal F_q$ is finite-dimensional,
letting $w\to\omega\in S$ gives
$h(\,\cdot\,,\omega)^q\in\mathcal F_q$.
Hence Lemma~\ref{32} gives \(\mathcal D_{q+1,z}h(z,\omega)^q=0\).
The operator $\mathcal D_{q+1}$ contains only holomorphic derivatives, while
$h(\omega,z)^{-m}$ is antiholomorphic in $z$. On every compact subset of
$\Omega$, the kernel and all derivatives occurring here are bounded
uniformly in $\omega\in S$. Since $f\in L^2(S)\subset L^1(S)$, we may
differentiate under the integral, obtaining \eqref{55}.
\end{proof}

\subsection{Radial tests for boundary regularity}

Proposition~\ref{51} reduces the possible parameters, and
Proposition~\ref{54} settles the cases $\sigma=0$ and
$\sigma\in\mathbb N$, $\theta\in-\mathbb N_0$.
We now prove two further cases of Theorem~\ref{3}:
$\theta=\frac nr$ with $\sigma\ne0$ and
$\sigma-\frac nr\notin\mathbb N_0$; and
$\sigma-\frac nr\in\mathbb N_0$ with
$\frac nr-\theta\in\mathbb N$ and $\theta\notin-\mathbb N_0$.
Both proofs use $K$-type projections and the radial
formulas of Lemma~\ref{43}. The case $\theta=\frac nr$ with
$\sigma-\frac nr\in\mathbb N_0$ is treated separately in the
tube-type and non-tube-type sections.

\begin{lemma}\label{551}
Let $\lambda>0$ and $\mu\in\mathbb C\setminus\{0\}$, with
$\lambda-\mu\notin-\mathbb N_0$. For $0<\rho<1$ near $0$,
\[
{}_2F_1(\lambda-\mu,1;\lambda+1;1-\rho)
=H(\rho)+
\begin{cases}
\displaystyle
\frac{\Gamma(\lambda+1)\Gamma(-\mu)}
     {\Gamma(\lambda-\mu)}
\rho^\mu(1-\rho)^{-\lambda},
&\mu\notin\mathbb N_0,\\[3mm]
\displaystyle
\frac{(-1)^{\mu+1}\Gamma(\lambda+1)}
     {\mu!\,\Gamma(\lambda-\mu)}
\rho^\mu\log\rho\,(1-\rho)^{-\lambda},
&\mu\in\mathbb N,
\end{cases}
\]
where $H$ is analytic near $\rho=0$. In both cases the
displayed coefficient is nonzero.
\end{lemma}

\begin{proof}
For $\mu\notin\mathbb Z$, the connection formula
\cite[formula~(15.8.4)]{DLMF}, with the normalization
\cite[formula~(15.1.2)]{DLMF}, gives
\[
{}_2F_1(\lambda-\mu,1;\lambda+1;1-\rho)
=\frac{\lambda}{\mu}
{}_2F_1(\lambda-\mu,1;1-\mu;\rho)
+\frac{\Gamma(\lambda+1)\Gamma(-\mu)}
       {\Gamma(\lambda-\mu)}
\rho^\mu(1-\rho)^{-\lambda}.
\]
The first term is analytic at $\rho=0$. For $\mu\in-\mathbb N$,
$1-\mu,\lambda-\mu>0$, so both parameter limits are finite
and the identity remains valid. For $\mu\in\mathbb N$,
\cite[formula~(15.8.10)]{DLMF} gives the logarithmic term
with an analytic remainder. Finally, $\Gamma$ has no zeros,
and the assumption on $\lambda-\mu$ keeps the denominator finite.
\end{proof}

\begin{theorem}\label{56}
Let $\sigma\in\mathbb C\setminus\{0\}$ and
$\sigma-\frac{n}{r}\notin\mathbb N_0$.  Put \(M=\max\{0,\lceil\operatorname{Re}\sigma\rceil\}\).
If $f\in L^2(S)$ and
$F=P_{\sigma,\frac{n}{r}}f\in C^M(\overline\Omega)$, then $F$ is holomorphic.
\end{theorem}
\begin{proof}
Let $E=\overline{\mathcal P_{(1)}}|_S$, and let $f_E$ be the orthogonal
projection of $f$ onto $E$. By the uniqueness statement following
\eqref{12}, $\mathcal P_{(1)}|_S$ occurs with multiplicity one in $L^2(S)$.
Complex conjugation preserves multiplicities, so $E$ is the full
$K$-isotypic component of its type. Lemma~\ref{29} gives
$P_{\sigma,n/r}f_E\in C^M(\overline\Omega)$. Since $P_{\sigma,n/r}$
intertwines the $K$-actions and $K$ preserves $C^M(\overline\Omega)$,
the space
\[
\{v\in E:P_{\sigma,n/r}v\in C^M(\overline\Omega)\}
\]
is $K$-invariant. If $f_E\ne0$, irreducibility of $E$ makes this space
equal to $E$.

By Lemma~\ref{16}, $\overline{\ell_c}|_S\in E$ for every minimal tripotent
$c$. If $f_E\ne0$, it follows that
$$P_{\sigma,n/r}(\overline{\ell_c})\in C^M(\overline\Omega).$$
We test this regularity on the radius $z=tc$. Put $\rho=1-t^2$.
By \eqref{46} and Euler's transformation
\cite[formula~(15.8.1)]{DLMF},
\[
P_{\sigma,n/r}(\overline{\ell_c})(tc)
=\frac{r\sigma}{n}\sqrt{1-\rho}\,
{}_2F_1\left(\frac nr-\sigma,1;\frac nr+1;1-\rho\right).
\]
The hypothesis $\sigma-\frac nr\notin\mathbb N_0$ gives
$\frac nr-\sigma\notin-\mathbb N_0$.
Lemma~\ref{551}, with $\lambda=\frac nr$ and $\mu=\sigma$,
gives the following expansion near $\rho=0$:
\[
\begin{aligned}
{}_2F_1\left(\frac nr-\sigma,1;\frac nr+1;1-\rho\right)
={}&H(\rho)+(1-\rho)^{-n/r}
\begin{cases}
\displaystyle
\frac{\Gamma(\frac nr+1)\Gamma(-\sigma)}
     {\Gamma(\frac nr-\sigma)}\,\rho^\sigma,
&\sigma\notin\mathbb N_0,\\[3mm]
\displaystyle
\frac{(-1)^{\sigma+1}\Gamma(\frac nr+1)}
     {\sigma!\,\Gamma(\frac nr-\sigma)}
     \,\rho^\sigma\log\rho,
&\sigma\in\mathbb N.
\end{cases}
\end{aligned}
\]
Here $H$ is analytic near $\rho=0$, and both displayed coefficients
are nonzero.

Since $(1-\rho)^{-n/r}$ and $\sqrt{1-\rho}$ are analytic
and nonzero at $\rho=0$, hence they do not remove the singularity
found above. If $\sigma\notin\mathbb N_0$ and
$\operatorname{Re}\sigma\le0$, then $M=0$ and $\rho^\sigma$ has no
finite limit at $0$. If $\sigma\notin\mathbb N_0$ and
$\operatorname{Re}\sigma>0$, the $M$-th derivative of $\rho^\sigma$
has a nonzero leading term proportional to $\rho^{\sigma-M}$,
which has no finite limit. Finally, if $\sigma\in\mathbb N$, then
$M=\sigma$ and the $M$-th derivative of $\rho^\sigma\log\rho$
diverges. Since $\rho=1-t^2$ is a smooth coordinate near $t=1$,
$P_{\sigma,n/r}(\overline{\ell_c})(tc)$ cannot be $C^M$ there.
This contradicts the regularity obtained from $f_E\ne0$, so
$f_E=0$.

Thus $f\perp\overline{\mathcal P_{(1)}}|_S$, or equivalently,
\[
\int_S f(\omega)p(\omega)\,du(\omega)=0
\qquad (p\in\mathcal P_{(1)}).
\]
By \eqref{9}, $h(0,0)=1$ and
$\bar\partial h(z,z)|_{z=0}=0$. Differentiating \eqref{20}
at $z=0$, only the terms with $\alpha=0$ and $\beta=(1)$
contribute to the antiholomorphic derivative. Thus
\[
\bar\partial F(0)
=
\sigma\sum_i
\left(\int_S f(\omega)p_{(1),i}(\omega)\,du(\omega)\right)
\bar\partial\overline{p_{(1),i}}(0)
=0.
\]
Now let $\widetilde g\in\widetilde G$, put $g=\pi(\widetilde g)$, and set \(F_{\widetilde g} = P_{\sigma,n/r} \pi_{\sigma,n/r}(\widetilde g)f\).
By Lemma~\ref{29}, $F_{\widetilde g}$ has the same boundary regularity,
so the preceding argument gives \(\bar\partial F_{\widetilde g}(0)=0\).
Formula~\eqref{27}
and the chain rule then give
$\bar\partial F(g^{-1}0)=0$: the Jacobian factor is
holomorphic and nonzero, and $(g^{-1})'(0)$ is invertible.
As $G$ acts transitively on $\Omega$, we obtain
$\bar\partial F=0$ on $\Omega$.
\end{proof}

\begin{theorem}\label{57}
Let $\sigma,\theta\in\mathbb C$ satisfy
\[
\sigma-\frac nr\in\mathbb N_0,
\qquad
\frac nr-\theta\in\mathbb N,
\qquad
\theta\notin-\mathbb N_0.
\]
If $f\in L^2(S)$ and \(F=P_{\sigma,\theta}f \in\mathcal R^\sigma(\overline\Omega)\),
then
\begin{equation}\label{58}
\mathcal D_{\sigma-\frac nr+1}
\left(
h(z,z)^{\theta-\frac nr}F
\right)=0.
\end{equation}
\end{theorem}

\begin{proof}
Put
\[
q=\sigma-\frac nr+1\in\mathbb N,
\qquad
j_0=\frac nr-\theta\in\mathbb N,
\qquad
H(z)=h(z,z)^{-j_0}F(z).
\]
We have $H=h(z,z)^{\sigma-j_0}(h(z,z)^{-\sigma}F)$.
Since $h(z,w)$ is antiholomorphic in $w$ and $h(z,0)=1$,
for every nonzero multi-index $\gamma$,
\[
\left.\partial_z^\gamma h(z,z)^{\sigma-j_0}\right|_{z=0}
=\left.\partial_z^\gamma h(z,0)^{\sigma-j_0}\right|_{z=0}
=0.
\]
Also $h(0,0)^{\sigma-j_0}=1$. The product rule therefore gives
$$p^*(\partial)H(0)=p^*(\partial)(h(z,z)^{-\sigma}F)(0)\ \
\mbox{for every}\  p\in\mathcal P_{(q)}.$$
Apply $p^*(\partial)$ to \eqref{20} and evaluate at $z=0$.
Terms with $\beta\ne0$ vanish because they still contain
$\bar z$. For the terms with $\beta=0$, the definition of the
Fischer inner product gives
\[
p^*(\partial)p_{\alpha,i}(0)
=\langle p_{\alpha,i},p\rangle_F.
\]
This vanishes unless $\alpha=(q)$, since
$p\in\mathcal P_{(q)}$ and the spaces $\mathcal P_\alpha$
are Fischer-orthogonal. For $\alpha=(q)$, expansion in the
Fischer-orthonormal basis gives
\[
\sum_i\langle p_{(q),i},p\rangle_F
\overline{p_{(q),i}(\omega)}=\overline{p(\omega)}.
\]
Using $(\theta)_{(q)}=(\theta)_q$, we obtain
\[
\bigl(\mathcal D_qH(0)\bigr)(p^*)
=(\theta)_q
\int_S f(\omega)\overline{p(\omega)}\,du(\omega),
\qquad p\in\mathcal P_{(q)}.
\]
The coefficient $(\theta)_q$ is nonzero by the assumption
$\theta\notin-\mathbb N_0$.

Let $E=\mathcal P_{(q)}|_S$, and let $f_E$ be the orthogonal
projection of $f$ onto $E$. By the uniqueness statement following
\eqref{12}, $E$ is the full $K$-isotypic component of its type.
The identity above shows that $\mathcal D_qH(0)\ne0$ implies
$f_E\ne0$. As in the proof of Theorem~\ref{56},
Lemma~\ref{29} and irreducibility of $E$ then give
$
P_{\sigma,\theta}E\subset\mathcal R^\sigma(\overline\Omega).
$
Lemma~\ref{16} therefore gives
\[
P_{\sigma,\theta}(\ell_c^q)
\in\mathcal R^\sigma(\overline\Omega)
\ \ 
\mbox{for every minimal tripotent}\ c.\]
By \eqref{47},
\[
P_{\sigma,\theta}(\ell_c^q)(tc)
=
\frac{(\theta)_q}{(\frac nr)_q}\,
t^q(1-t^2)^{j_0}
{}_2F_1(1,j_0;\sigma+1;t^2).
\]
Put $\rho=1-t^2$. If $\sigma\in\mathbb N$, then
$\frac nr\in\mathbb N$, and $\theta\notin-\mathbb N_0$
forces $\theta\in\mathbb N$. Thus
$\sigma-j_0=q-1+\theta\in\mathbb N$.
If $\sigma\notin\mathbb N$, then $\sigma-j_0$ is not an
integer, since $j_0\in\mathbb N$. Thus Lemma~\ref{551}
applies with $\lambda=\sigma$ and $\mu=\sigma-j_0$;
in particular, $\lambda-\mu=j_0>0$.
Using $t^q=(1-\rho)^{q/2}$ in the radial formula gives
\[
\begin{aligned}
P_{\sigma,\theta}(\ell_c^q)(tc)
={}&R(\rho)+\frac{(\theta)_q}{(\frac nr)_q}
(1-\rho)^{q/2-\sigma}
\begin{cases}
\displaystyle
\frac{\Gamma(\sigma+1)\Gamma(j_0-\sigma)}
     {\Gamma(j_0)}\,\rho^\sigma,
&\sigma\notin\mathbb N,\\
\displaystyle
\frac{(-1)^{\sigma-j_0+1}\Gamma(\sigma+1)}
     {(\sigma-j_0)!\,\Gamma(j_0)}
     \,\rho^\sigma\log\rho,
&\sigma\in\mathbb N.
\end{cases}
\end{aligned}
\]
Here $R$ is analytic near $\rho=0$, and the coefficient of the
singular term in each case is nonzero. Since
$(1-\rho)^{q/2-\sigma}$ is analytic and nonzero at $0$,
when $\sigma\notin\mathbb N$ the term
$\rho^\sigma(1-\rho)^{q/2-\sigma}$ does not belong to
$C^{\lfloor\sigma\rfloor,\varepsilon}$ for any
$\varepsilon>\sigma-\lfloor\sigma\rfloor$.
When $\sigma\in\mathbb N$, the term
$\rho^\sigma\log\rho\,(1-\rho)^{q/2-\sigma}$ is not $C^\sigma$.
Thus the radial function does not have the regularity required by
$\mathcal R^\sigma$.
This contradicts the regularity of
$P_{\sigma,\theta}(\ell_c^q)$. Therefore
$\mathcal D_qH(0)=0$.

For $\widetilde g\in\widetilde G$, set
\[
F_{\widetilde g}
=P_{\sigma,\theta}\pi_{\sigma,\theta}(\widetilde g)f,
\qquad
H_{\widetilde g}
=h(z,z)^{-j_0}F_{\widetilde g}.
\]
Lemma~\ref{29} allows the same argument to be applied to
$F_{\widetilde g}$, so $\mathcal D_qH_{\widetilde g}(0)=0$.
Writing $g=\pi(\widetilde g)$, formula~\eqref{380} with
$L=q-1$ gives
\[
\mathcal D_qH_{\widetilde g}(0)
=
\overline{J_{\widetilde g^{-1}}(0)}^{\,j_0/p_\Omega}
J_{\widetilde g^{-1}}(0)^{-(q-1)/p_\Omega}
B_g(0)(\mathcal D_qH)(g^{-1}0).
\]
The Jacobian factors are nonzero and $B_g(0)$ is invertible.
Thus $(\mathcal D_qH)(g^{-1}0)=0$ for every $g\in G$.
Transitivity of $G$ gives $\mathcal D_qH=0$ on $\Omega$,
which is \eqref{58}.
\end{proof}

To treat the remaining case $\theta=\frac nr$ and
$\sigma-\frac nr\in\mathbb N_0$, define, for $q\ge1$,
\begin{equation}\label{59}
\mathcal M_q
=
\operatorname{span}
\left\{
\left(
p\bar z_j-\frac{\partial_jp}{\frac nr+q-1}
\right)\Big|_S:
p\in\mathcal P_{(q)},\ 1\le j\le n
\right\}.
\end{equation}

\begin{lemma}
For $q\ge1$, the space $\mathcal M_q$ is $K$-invariant,
and the following decomposition is orthogonal:
\begin{equation}\label{60}
\operatorname{span}
\left\{
p\overline\ell|_S:
p\in\mathcal P_{(q)},\ \ell\in\mathcal P_{(1)}
\right\}
=
\mathcal P_{(q-1)}|_S\oplus\mathcal M_q.
\end{equation}
\end{lemma}

\begin{proof}
Let $E$ denote the space on the left of \eqref{60}.
Both polynomial factors are $K$-invariant, hence so is $E$.
By Lemma~\ref{16} and
$\partial_j\ell_\xi^q=q\overline{\xi_j}\ell_\xi^{q-1}$,
\[
\mathcal P_{(q-1)}
=
\operatorname{span}\{\ell_\xi^{q-1}:\operatorname{rank}\xi=1\}
=
\operatorname{span}\{\partial_jp:
p\in\mathcal P_{(q)},\ 1\le j\le n\}.
\]
Lemma~\ref{21} and the definition \eqref{59} therefore give
\[
\Pi_HE=\mathcal P_{(q-1)}|_S,
\qquad
(I-\Pi_H)E=\mathcal M_q.
\]

Let $\tau$ be the $K$-type of $\mathcal P_{(q-1)}|_S$.
By the uniqueness statement following \eqref{12},
$\Pi_\tau^S$ is the orthogonal projection onto this space.
For $v\in E$, we have $\Pi_Hv\in\mathcal P_{(q-1)}|_S$
and $v-\Pi_Hv\perp H^2(S)$, so
$\Pi_\tau^Sv=\Pi_Hv$. Since $E$ is $K$-invariant,
the character formula gives $\Pi_\tau^SE\subset E$.
Thus $\Pi_HE\subset E$, and
\[
E=\Pi_HE\oplus(I-\Pi_H)E
 =\mathcal P_{(q-1)}|_S\oplus\mathcal M_q.
\]
The sum is orthogonal because $\Pi_H$ is an orthogonal
projection. Finally, $\Pi_H$ commutes with $K$, so
$\mathcal M_q=(I-\Pi_H)E$ is $K$-invariant.
\end{proof}

\begin{lemma}\label{61}
Assume that \(\sigma-\frac nr\in\mathbb N_0, \quad q=\sigma-\frac nr+1\).
For $f\in L^2(S)$,  let \(F=P_{\sigma,\frac nr}f \).
Then, for $p\in\mathcal P_{(q)}$,
\begin{equation}\label{62}
\bar\partial_j\bigl(p^*(\partial)F\bigr)(0)
=
\sigma\left(\frac nr\right)_{(q)}
\left\langle
f,\,
p\bar z_j-\frac{\partial_jp}{\sigma}
\right\rangle_S.
\end{equation}
In particular, \(f\perp\mathcal M_q \quad\Longrightarrow\quad \bar\partial\mathcal D_qF(0)=0\).
\end{lemma}

\begin{proof}
For $z,w$ near $0$, put
\[
\widetilde F(z,w)
=
h(z,w)^\sigma
\int_S
h(z,\omega)^{-\frac nr}
h(\omega,w)^{-\sigma}
f(\omega)\,du(\omega).
\]
Then $\widetilde F(z,z)=F(z)$. By \eqref{13} with $u=1$ and the
degree-one case of \eqref{14},
$$h(x,w)^{-1}=1+\sum_{k=1}^n x_k\bar w_k+O(|w|^2),\ \ 
\mbox{as}\ w\to0.$$ Since $h(x,0)=1$, differentiating the above formula gives
$$\bar\partial_{w_j}h(x,w)\big|_{w=0}=-x_j.$$
As $f\in L^2(S)\subset L^1(S)$, we may differentiate under the
integral near the origin and obtain
\[
\bar\partial_{w_j}\widetilde F(z,0)
=
\sigma\left[
\int_S h(z,\omega)^{-\frac nr}\omega_jf(\omega)\,du(\omega)
-z_j\int_S h(z,\omega)^{-\frac nr}f(\omega)\,du(\omega)
\right].
\]
For $p\in\mathcal P_{(q)}$, the Faraut--Kor\'anyi expansion and Fischer
adjointness give
\[
p^*(\partial_z)h(z,\omega)^{-\frac nr}\Big|_{z=0}
=
\left(\frac nr\right)_{(q)}\overline{p(\omega)}
\]
and
\[
p^*(\partial_z)
\left[z_jh(z,\omega)^{-\frac nr}\right]\Big|_{z=0}
=
\left(\frac nr\right)_{(q-1)}
\overline{\partial_jp(\omega)}.
\]
Since \(\left(\frac nr\right)_{(q)} = \sigma\left(\frac nr\right)_{(q-1)}\),
restriction to the diagonal gives \eqref{62}.
\end{proof}

\section{The tube type case}

\label{63}

Assume throughout this section that $\Omega$ is of tube type.
The unit disk is the only rank-one tube-type domain, so $r>1$.
Thus \(b=0\) and \(\frac nr=1+\frac a2(r-1)\).

\begin{lemma}
\label{64}
For every $q\ge1$, the space $\mathcal M_q$ in \eqref{59} is
irreducible and is the full $K$-isotypic component of its type in
$L^2(S)$.
\end{lemma}

\begin{proof}
For the function spaces below, their product means the span of all
pointwise products on $S$. Write $1^d$ for a tuple of $d$ ones,
with $1^0$ empty.

For tube-type domains, $S=K/K_1$ is a compact symmetric space;
see \cite[Section~1]{ZhangRec95}. Hence $L^2(S)$ contains at most
one copy of each irreducible $K$-type.

By \cite[Section~2, proof of Lemma~2]{ZhangRec95},
$\mathcal P_{(1^r)}|_S$ is one-dimensional. A nonzero element
of its conjugate is nowhere zero on $S$, since $K$ acts
transitively on $S$. Thus pointwise multiplication by this
element realizes
$\overline{\mathcal P_{(1^r)}|_S}\,
 \mathcal P_{(1^{r-1},0)}|_S$
as the product of an irreducible $K$-type with a
one-dimensional character, and this product is irreducible
\cite[Section~2.3]{FultonHarris}.

The one-dimensional $K$-type
$\overline{\mathcal P_{(1^r)}|_S}$ has signature
$(-1,\ldots,-1)$. Hence
$\overline{\mathcal P_{(1^r)}|_S}\,
\mathcal P_{(1^{r-1},0)}|_S$
has signature
$(1^{r-1},0)+(-1,\ldots,-1)=(0^{r-1},-1)$
by \cite[Lemma~3]{ZhangRec95}.
The $K$-type $\overline{\mathcal P_{(1)}|_S}$
has the same signature. Since $L^2(S)$ is multiplicity-free,
\[
\overline{\mathcal P_{(1)}|_S}
=
\overline{\mathcal P_{(1^r)}|_S}\,
\mathcal P_{(1^{r-1},0)}|_S.
\]

For the pointwise product of $\mathcal P_{(q)}|_S$ and
$\mathcal P_{(1^{r-1},0)}|_S$,
\cite[Theorem~1(1)]{ZhangRec95} allows only the two
$K$-types displayed below. The recurrence formula
\cite[Theorem~1(2)]{ZhangRec95} shows that the product
of the normalized $K_1$-fixed vectors of these two spaces
has a nonzero component in each type. Since the pointwise
product space is $K$-invariant, it contains both full types:
\[
\mathcal P_{(q)}|_S\,
\mathcal P_{(1^{r-1},0)}|_S
=
\mathcal P_{(q,1^{r-1})}|_S
\oplus
\mathcal P_{(q+1,1^{r-2},0)}|_S.
\]
Multiplying by $\overline{\mathcal P_{(1^r)}|_S}$
and using the preceding equality gives
\[
\begin{aligned}
\mathcal P_{(q)}|_S\,\overline{\mathcal P_{(1)}|_S}
&=
\overline{\mathcal P_{(1^r)}|_S}\,
\mathcal P_{(q,1^{r-1})}|_S\oplus
\overline{\mathcal P_{(1^r)}|_S}\,
\mathcal P_{(q+1,1^{r-2},0)}|_S\\
&=
\mathcal P_{(q-1)}|_S
\oplus
\overline{\mathcal P_{(1^r)}|_S}\,
\mathcal P_{(q+1,1^{r-2},0)}|_S.
\end{aligned}
\]
The last equality again follows from the signatures and
multiplicity one in $L^2(S)$. Comparing with \eqref{60},
we obtain
\[
\mathcal M_q
=
\overline{\mathcal P_{(1^r)}|_S}\,
\mathcal P_{(q+1,1^{r-2},0)}|_S.
\]
The right-hand side is irreducible, has signature
$(q,0^{r-2},-1)$, and is the full $K$-isotypic
component of that type in $L^2(S)$.
\end{proof}

For the rest of this section, fix $\sigma$ with
$\sigma-\frac nr\in\mathbb N_0$ and put $q=\sigma-\frac nr+1$.
We now compute the Poisson transform of one nonzero vector in
$\mathcal M_q$.

\begin{lemma}
\label{65} For a minimal tripotent $c$,
let \(f_{c} = \ell_{c}^{q}\overline{\ell_{c}} - \frac{q}{\sigma}\ell_{c}^{q-1}\).
Then $f_{c}\in\mathcal{M}_{q}, $ and for $0\le t<1$,
\begin{equation}
\label{66}P_{\sigma,\frac nr}f_{c}(tc) = \frac{\frac
nr-1}{\sigma+1} t^{q+1} {}_{2}F_{1}(1,1;\sigma+2;t^{2}).
\end{equation}
In particular, $P_{\sigma,\frac nr}f_{c} \notin\mathcal{R}^{\sigma}%
(\overline\Omega). $
\end{lemma}

\begin{proof}
Since $\ell_c^q\in\mathcal P_{(q)}$, summing the corresponding
generators in \eqref{59} with coefficients $c_j$ gives
$f_c\in\mathcal M_q$; here we use $\|c\|=1$ and
$\sigma=\frac nr+q-1$.

Put $x=t^2$ and first suppose $0<t<1$. Using
$t\overline{\ell_c}=1-(1-t\overline{\ell_c})$ in the
Poisson kernel \eqref{8}, together with \eqref{42}, gives
\[
P_{\sigma,\frac nr}(\ell_c^q\overline{\ell_c})(tc)
=
\frac{1}{t}\left(
P_{\sigma,\frac nr}(\ell_c^q)(tc)
-
P_{\sigma,\frac nr-1}(\ell_c^q)(tc)
\right).
\]
Since $\sigma=\frac nr+q-1$, Lemma~\ref{43}, in the form
\eqref{47}, gives
\[
\begin{aligned}
&P_{\sigma,\frac nr}(\ell_c^q)(tc)=t^q,\ \ 
P_{\sigma,\frac nr}(\ell_c^{q-1})(tc)=t^{q-1},\\
&P_{\sigma,\frac nr-1}(\ell_c^q)(tc)
=
\frac{\frac nr-1}{\sigma}\,
t^q(1-x)\,{}_2F_1(1,1;\sigma+1;x).
\end{aligned}
\]
Therefore, using $\sigma-q=\frac nr-1$,
\[
P_{\sigma,\frac nr}f_c(tc)
=
\frac{\frac nr-1}{\sigma}\,t^{q-1}
\left[1-(1-x)\,{}_2F_1(1,1;\sigma+1;x)\right].
\]
Comparing coefficients in the Gauss series gives
\[
1-(1-x)\,{}_2F_1(1,1;\sigma+1;x)
=
\frac{\sigma x}{\sigma+1}\,
{}_2F_1(1,1;\sigma+2;x).
\]
This proves \eqref{66} for $t>0$, and continuity gives it at $t=0$.

Since $r>1$, we have
\(\frac nr=1+\frac a2(r-1)>1\).
Put $\rho=1-t^2$. Lemma~\ref{551}, with
$\lambda=\sigma+1$ and $\mu=\sigma$, gives
\[
P_{\sigma,\frac nr}f_c(tc)
=
H(\rho)
+
\left(\frac nr-1\right)
(1-\rho)^{(q+1)/2-\sigma-1}
\begin{cases}
\Gamma(\sigma+1)\Gamma(-\sigma)\rho^\sigma,
&\sigma\notin\mathbb N,\\[2mm]
(-1)^{\sigma+1}\rho^\sigma\log\rho,
&\sigma\in\mathbb N,
\end{cases}
\]
where $H$ is analytic near $\rho=0$. The coefficient
of the singular term is nonzero in either case.
Therefore $P_{\sigma,\frac nr}f_c\notin
\mathcal R^\sigma(\overline\Omega)$ and the proof is complete.
\end{proof}

\begin{theorem}
\label{67} Let $f\in L^2(S)$ and put \(F=P_{\sigma,\frac nr}f\).
If $F\in\mathcal{R}^{\sigma}(\overline\Omega), $ then
\begin{equation}
\label{68}\bar\partial\mathcal{D}_{\sigma-\frac nr+1}F=0.
\end{equation}

\end{theorem}

\begin{proof}
By Lemma~\ref{64}, $\mathcal M_q$ is
irreducible and is the full $K$-isotypic component of its type.
Let $f_{\mathcal M}$ be the orthogonal projection of $f$
onto $\mathcal M_q$. By Lemma~\ref{29},
$P_{\sigma,\frac nr}f_{\mathcal M}
\in\mathcal R^\sigma(\overline\Omega)$.
As in the proof of Theorem~\ref{57}, irreducibility of
$\mathcal M_q$ implies that, if $f_{\mathcal M}\ne0$, then
\[
P_{\sigma,\frac nr}\mathcal M_q
\subset\mathcal R^\sigma(\overline\Omega).
\]
This contradicts Lemma~\ref{65}, since $f_c\in\mathcal M_q$.
Hence $f_{\mathcal M}=0$, or equivalently
$f\perp\mathcal M_q$.

By Lemma~\ref{61}, the orthogonality
$f\perp\mathcal M_q$ gives
$
\bar\partial\mathcal D_qF(0)=0.
$

Now let $\widetilde g\in\widetilde G$ and put
$
F_{\widetilde g}
=
P_{\sigma,\frac nr}
\pi_{\sigma,\frac nr}(\widetilde g)f.
$
By Lemma~\ref{29},
$
F_{\widetilde g}\in\mathcal R^\sigma(\overline\Omega).
$
Applying the preceding argument to
$\pi_{\sigma,\frac nr}(\widetilde g)f$ gives
$
\bar\partial\mathcal D_qF_{\widetilde g}(0)=0.
$
Since this holds for every $\widetilde g\in\widetilde G$ and
$
q=\sigma-\frac nr+1,
$
Lemma~\ref{35} gives \(\bar\partial\mathcal D_qF=0 \ \text{on }\Omega\).
This is \eqref{68}.
\end{proof}

\section{The non-tube type case}

\label{69}

Throughout this section, assume that $\Omega$ is not of tube type,
so $b>0$, and fix $\sigma$ with
$\sigma-\frac nr\in\mathbb N_0$.
Fix an arbitrary tripotent $c$ of rank $r-1$, taking $c=0$ if
$r=1$, and set $\Omega_0(c)=\Omega\cap V_0(c)$. The space $V_0(c)$
has rank one and complex dimension $b+1$
\cite[Section~2, formulas~(4)--(6)]{KZ}, so $\Omega_0(c)$ is its
unit ball.

We first record the normalization of the kernel. Put
\[
\mathcal{K}_{\sigma}(z,\omega) = h(z,z)^{\sigma}h(z,\omega)^{-\frac nr}
h(\omega,z)^{-\sigma}.
\]

For fixed $z\in\Omega$, the function
$h(\cdot,z)^{-\sigma}$ is holomorphic in a neighborhood of
$\overline\Omega$, and hence its restriction to $S$ belongs to $H^2(S)$.
The reproducing formula \eqref{151} therefore gives
\[
\int_S
h(z,\omega)^{-\frac nr}h(\omega,z)^{-\sigma}\,du(\omega)
=
h(z,z)^{-\sigma}.
\]
Thus
\begin{equation}\label{72}
\int_S\mathcal K_\sigma(z,\omega)\,du(\omega)=1.
\end{equation}

\begin{lemma}
\label{73}
For every $f\in C(S)$,
\[
\lim_{t\uparrow1}
\sup_{\omega_0\in S}
\left|
P_{\sigma,\frac nr}f(t\omega_0)-f(\omega_0)
\right|=0.
\]
For every $f\in L^2(S)$, \(P_{\sigma,\frac nr}f(t\,\cdot)\longrightarrow f \quad\text{in }L^2(S)\)
as $t\uparrow1$. In particular, if $P_{\sigma,\frac nr}f$ extends
continuously to $\overline\Omega$, then $f$ has a continuous representative
and \(f=\left.P_{\sigma,\frac nr}f\right|_S\).
\end{lemma}

\begin{proof}
Write \(T_tf(\omega_0)=P_{\sigma,\frac nr}f(t\omega_0)\).
We use the uniform boundary limit for the Poisson--Szeg\H{o}
transform \(P_{\frac nr,\frac nr}\) on $C(S)$, given by
\cite[Theorem~4.3]{BK07} at $s=n/r$, with the normalization
in \eqref{72}.

We first compare the kernels.  For $\omega_0,\omega\in S$, put
$p(\zeta)=h(\zeta\omega_0,\omega)$.  By \eqref{13} with $u=-1$,
$p$ is a polynomial of degree at most $r$, and $p(0)=1$.
The Szeg\H{o}-kernel formula \cite[(3.18)]{Koranyi} gives
$p(\zeta)\ne0$ for $|\zeta|<1$.  Thus
$p(\zeta)=\prod_{j=1}^{d}(1-\lambda_j\zeta)$, where $d\le r$
and $|\lambda_j|\le1$.  For $0\le t<1$ this gives
\[
|h(t\omega_0,\omega)|\ge(1-t)^r,
\qquad h(t\omega_0,t\omega_0)=(1-t^2)^r,
\]
the second identity following from \eqref{9}.  Since
$\sigma-\frac nr\in\mathbb N_0$, we obtain
\begin{equation}\label{501}
|\mathcal K_\sigma(t\omega_0,\omega)|
=
\left(\frac{h(t\omega_0,t\omega_0)}{|h(t\omega_0,\omega)|}\right)^{\sigma-\frac nr}
\mathcal K_{\frac nr}(t\omega_0,\omega)
\le C\,\mathcal K_{\frac nr}(t\omega_0,\omega),
\qquad C=2^{r(\sigma-\frac nr)}.
\end{equation}

The positivity of $\mathcal K_{\frac nr}$, \eqref{72}, and
\eqref{501} give
\[
\|T_tf\|_\infty\le C\|f\|_\infty,
\qquad f\in C(S),\quad 0\le t<1.
\]
To prove convergence on a dense subspace of $C(S)$, fix $e_0\in S$
and put \(d_{e_0}(\omega)=\|\omega-e_0\|\).  By the cited
Poisson--Szeg\H{o} boundary limit,
\[
P_{\frac nr,\frac nr}d_{e_0}(te_0)
\longrightarrow d_{e_0}(e_0)=0.
\]
The $K$-invariance of the base kernel and $du$, and the transitive
unitary action of $K$ on $S$, show that
\[
\int_S\mathcal K_{\frac nr}(t\omega_0,\omega)
\|\omega-\omega_0\|\,du(\omega)
=P_{\frac nr,\frac nr}d_{e_0}(te_0)
\]
for every $\omega_0\in S$.  If $f$ is Lipschitz on $S$ with
constant $L_f$, then
\eqref{72} and \eqref{501} give
\[
\begin{aligned}
|T_tf(\omega_0)-f(\omega_0)|
&\le C\int_S\mathcal K_{\frac nr}(t\omega_0,\omega)
|f(\omega)-f(\omega_0)|\,du(\omega)\\
&\le CL_f
P_{\frac nr,\frac nr}d_{e_0}(te_0).
\end{aligned}
\]
Thus $T_tf\to f$ uniformly for Lipschitz $f$.  These functions are
dense in $C(S)$; for $f\in C(S)$ and Lipschitz $g$,
\[
\|T_tf-f\|_\infty
\le(C+1)\|f-g\|_\infty+\|T_tg-g\|_\infty.
\]
The uniform bound on $T_t$ now gives $T_tf\to f$ uniformly for
every $f\in C(S)$.

By \eqref{72}, for every \(\omega_0\in S\),
\[
\int_S\mathcal K_{\frac nr}(t\omega_0,\omega)\,du(\omega)=1.
\]
For fixed \(t\), $K$-invariance makes
\[
A_t(\omega):=
\int_S\mathcal K_{\frac nr}(t\omega_0,\omega)\,du(\omega_0)
\]
constant on \(S\). Since \(du(S)=1\), Fubini's theorem and the
preceding identity give
\[
A_t(\omega)
=\int_S A_t(\eta)\,du(\eta)
=\int_S\int_S
  \mathcal K_{\frac nr}(t\omega_0,\eta)
  \,du(\eta)\,du(\omega_0)
=1.
\]
By \eqref{501} and Cauchy--Schwarz with respect to the probability
measure \(\mathcal K_{\frac nr}(t\omega_0,\omega)\,du(\omega)\),
\[
|T_tf(\omega_0)|^2
\le C^2\int_S\mathcal K_{\frac nr}(t\omega_0,\omega)
|f(\omega)|^2\,du(\omega).
\]
Integrating in \(\omega_0\) and using \(A_t(\omega)=1\) gives
\[
\|T_tf\|_{L^2(S)}\le C\|f\|_{L^2(S)},
\qquad 0\le t<1.
\]
For \(f\in L^2(S)\) and \(g\in C(S)\), it follows that
\[
\|T_tf-f\|_{L^2(S)}
\le (C+1)\|f-g\|_{L^2(S)}+\|T_tg-g\|_{L^2(S)}.
\]
The last term tends to zero by the uniform convergence proved above.
Since \(C(S)\) is dense in \(L^2(S)\), we obtain \(T_tf\to f\) in
\(L^2(S)\).

Finally, if \(F=P_{\sigma,\frac nr}f\) is continuous on
\(\overline\Omega\), then \(F(t\,\cdot)\to F|_S\) uniformly, while
the result just proved gives \(F(t\,\cdot)\to f\) in \(L^2(S)\).
Hence \(f=F|_S\) almost everywhere.
\end{proof}

By Lemma~\ref{40}, every $F=P_{\sigma,\frac nr}f$ with $f\in L^2(S)$
satisfies
\begin{equation}
\label{74}
\Delta_{\Omega}F
+\left(\sigma-\frac nr\right)
\sum_j\overline{(z-Q(z)\bar z)_j}\,\bar\partial_jF=0.
\end{equation}

For $u\in\partial\Omega_0(c)$, the orthogonal sum $c+u$ is a
tripotent of rank $r$, hence maximal in $V$. Thus
\(c+\partial\Omega_0(c)\subset S\)
\cite[Theorem~4.2]{Clerc}.
More generally, the spectral decomposition of $c+w$ gives
\(c+\overline{\Omega_0(c)}\subset\overline\Omega\).

Let $B_0$ be the Bergman operator of the Jordan triple subsystem
$V_0(c)$. Thus
$$B_0(w,\bar v)=B(w,\bar v)|_{V_0(c)}\ \  \mbox{for}\ w,v\in V_0(c).$$
In the inherited Hermitian inner product, the rank-one Jordan triple
product is
\[
\{w,\bar v,x\}
=\langle w,v\rangle x+\langle x,v\rangle w.
\]
Since \(Q(w)\bar v=\frac12\{w,\bar v,w\}\), we obtain
\begin{equation}\label{75}
Q(w)\bar v=\langle w,v\rangle w,
\qquad
B_0(w,\bar w)
=(1-\|w\|^2)(I-ww^*),
\end{equation}
where \(ww^*x=\langle x,w\rangle w\). In particular,
\(w-Q(w)\bar w=(1-\|w\|^2)w\).

\begin{lemma}
\label{76}
For every $w\in\Omega_0(c)$,
\begin{equation}
\label{77}
\begin{aligned}
B(c+w,\overline{c+w})|_{V_0(c)}
=B_0(w,\bar w),\ \ \ 
B(c+w,\overline{c+w})|_{V_2(c)\oplus V_1(c)}
=0.
\end{aligned}
\end{equation}
Also,
\begin{equation}
\label{78}
(c+w)-Q(c+w)\overline{(c+w)}
=
w-Q(w)\bar w.
\end{equation}
\end{lemma}

\begin{proof}
For $r=1$, we have $c=0$, and both identities are immediate.
Assume $r\ge2$. Write
\[
c=c_1+\cdots+c_{r-1},\qquad w=\rho c_r,\quad 0\le\rho<1,
\]
where $(c_1,\ldots,c_r)$ is a Jordan frame. When $w=0$, choose
$c_r$ arbitrarily in $V_0(c)$.

For the spaces $V_{ij}$ in \eqref{502},
\cite[formula~(23)]{EYZ} gives
\[
B(c+w,\overline{c+w})|_{V_{ij}}
=(1-|t_i|^2)(1-|t_j|^2)I|_{V_{ij}},
\quad
t_0=0,\quad t_1=\cdots=t_{r-1}=1,\quad t_r=\rho.
\]
The formula also holds when $t_i=1$, since both sides are polynomial
in the spectral coefficients.  By \eqref{503}, every summand of
$V_1(c)\oplus V_2(c)$ has an index among $1,\ldots,r-1$, so the
operator vanishes there.  Also by \eqref{503},
$V_0(c)=V_{rr}\oplus V_{0r}$, and the restrictions to these two
spaces are respectively $(1-\rho^2)^2I$ and $(1-\rho^2)I$.
These agree with the restrictions of $B_0(w,\bar w)$ in
\eqref{75}, proving \eqref{77}.

Since  \(c\) is
orthogonal to \(c_r\), and both are
tripotents, direct computations show that
\[
Q(c+w)\overline{(c+w)}
=Q(c+\rho c_r)\overline{(c+\rho c_r)}
=c+\rho^3c_r,
\qquad
Q(w)\bar w=\rho^3c_r.
\]
Subtracting gives \eqref{78} and the proof is complete.
\end{proof}

For $f\in L^2(S)$ with $F=P_{\sigma,\frac nr}f\in C^2(\overline\Omega)$,
set
\[
F_c(w)=F(c+w), \qquad w\in\overline{\Omega_0(c)}.
\]
We now derive the equation satisfied by $F_c$. Since
$F\in C^2(\overline\Omega)$ and the coefficients in \eqref{74}
are polynomial in $z$ and $\bar z$, that equation extends by
continuity to $c+w$ for every $w\in\Omega_0(c)$.

Choose orthonormal coordinates adapted to
$V=V_2(c)\oplus V_1(c)\oplus V_0(c)$, with the last $b+1$
coordinates on $V_0(c)$. By \eqref{77}, the Bergman operator
vanishes on $V_2(c)\oplus V_1(c)$ and restricts to
$B_0(w,\bar w)$ on $V_0(c)$. Equations \eqref{78} and
\eqref{75} give
\[
(c+w)-Q(c+w)\overline{(c+w)}
=(1-\|w\|^2)w.
\]
The derivatives of $F$ in the $V_0(c)$ directions at $c+w$
are the corresponding derivatives of $F_c$ at $w$.
Thus \eqref{74} at $c+w$ becomes
\[
\sum_{i,j=1}^{b+1}
\bigl(B_0(w,\bar w)\bigr)_{i\bar j}
\partial_i\bar\partial_jF_c(w)
+
\left(\sigma-\frac nr\right)(1-\|w\|^2)
\sum_{j=1}^{b+1}\bar w_j\bar\partial_jF_c(w)
=0.
\]
Substituting \eqref{75} and dividing by $1-\|w\|^2$ yields
\begin{equation}\label{80}
\sum_{i,j=1}^{b+1}
(\delta_{ij}-w_i\bar w_j)
\partial_i\bar\partial_jF_c(w)
+
\left(\sigma-\frac nr\right)
\sum_{j=1}^{b+1}\bar w_j\bar\partial_jF_c(w)
=0.
\end{equation}

We study the boundary values of $F_c$ on
$\partial\Omega_0(c)$. As noted above, $\Omega_0(c)$ is the unit ball of  $V_0(c)$. An orthonormal
basis identifies its boundary with $\partial\mathbb B^{b+1}$. For $p,q\in\mathbb N_0$, consider polynomials $Y(w,\bar w)$
that are homogeneous of degree $p$ in $w$ and degree $q$ in
$\bar w$, and satisfy
\[
\sum_{j=1}^{b+1}\partial_j\bar\partial_jY=0.
\]
Their restrictions to $\partial\mathbb B^{b+1}$ form
$H^{p,q}(\mathbb C^{b+1})$, the space of complex spherical
harmonics of bidegree $(p,q)$.
The spherical harmonic decomposition is orthogonal and gives
\cite{ABC}
\[
L^2(\partial\mathbb B^{b+1})
=
\widehat{\bigoplus}_{p,q\ge0}H^{p,q}(\mathbb C^{b+1}).
\]

We also write $H^{p,q}(V_0(c))$ for this space on
$\partial\Omega_0(c)$. Let $\Pi_{H^{p,q}(V_0(c))}$ denote the
orthogonal projection of $L^2(\partial\Omega_0(c))$ onto  $H^{p,q}(V_0(c))$.
When $c$ is fixed, we abbreviate this to $\Pi_{H^{p,q}}$.

\begin{lemma}
\label{81}
If $f\in C(S)$ and
$
F=P_{\sigma,\frac nr}f \in C^{b+1+\sigma-\frac nr}(\overline\Omega).
$
Then
\[
\Pi_{H^{\sigma-\frac nr+1,1}(V_0(c))}
\bigl(f(c+\cdot)\bigr)=0.
\]
\end{lemma}

\begin{proof}
We use the spherical-harmonic method of Graham and Chen--Li
\cite{Graham,ChenLi}.

Put $p=\sigma-\frac nr+1\in\mathbb N$ and
$m=b+p\in\mathbb N$. On $\partial\Omega_0(c)$, write
$g_c(u)=f(c+u)$. We show that its
$H^{p,1}(\mathbb C^{b+1})$ component vanishes.

For $0<t<1$ and $\zeta\in\partial\Omega_0(c)$, set
\[
U_c(t\zeta)
=
\Pi_{H^{p,1}}\bigl[F_c(t\,\cdot)\bigr](\zeta).
\]
By Lemma~\ref{73}, $F_c|_{\partial\Omega_0(c)}=g_c$.
Hence
\[
U_c|_{\partial\Omega_0(c)}
=
\Pi_{H^{p,1}(V_0(c))}g_c
=
\Pi_{H^{p,1}(V_0(c))}\bigl(f(c+\cdot)\bigr).
\] This projection is a
compact-group average and preserves the assumed $C^m$ regularity
on  $\overline{\Omega_0(c)}$. The operator in \eqref{80} commutes with the
$U(b+1)$-action; hence $U_c$ also satisfies \eqref{80}.

Choose homogeneous harmonic polynomials $(Y_\ell)$ whose restrictions
form an orthonormal basis of $H^{p,1}(\mathbb C^{b+1})$ on the unit
sphere. For $w\ne0$, we can write $U_c$ as
\[
\sum_\ell Y_\ell(w)\Phi_\ell(x),
\qquad x=\|w\|^2.
\]
For a harmonic polynomial $Y$ of bidegree $(p,1)$,
$\sum_jw_j\partial_jY=pY$ and
$\sum_j\bar w_j\bar\partial_jY=Y$. Using these identities and
$\sigma-\frac nr=p-1$ in \eqref{80}, then using the linear
independence of the $Y_\ell$ on each sphere, we find that each
$\Phi_\ell$ satisfies
\begin{equation}\label{83}
x(1-x)\Phi''+(m+2-3x)\Phi'-\Phi=0.
\end{equation}
This is Gauss's hypergeometric equation with parameters
$1,1$ and $m+2$.

Since the restrictions of the $Y_\ell$ are orthonormal, therefore
\[
t^{p+1}\Phi_\ell(t^2)
=
\int_{\partial\Omega_0(c)}
U_c(t\zeta)\overline{Y_\ell(\zeta)}\,d\zeta,
\qquad 0<t<1,
\]
where $d\zeta$ is normalized surface measure. Because
$U_c\in C^m(\overline{\Omega_0(c)})$, the integral is bounded
as $t\downarrow0$ and is $C^m$ up to $t=1$. The change of
variable $x=t^2$ is smooth and invertible near $t=1$.
Consequently,
\[
x^{(p+1)/2}\Phi_\ell(x)=O(1)\quad(x\downarrow0),
\qquad
\Phi_\ell\in C^m\text{ near }x=1.
\]

At $x=0$, equation~\eqref{83} has the regular solution
${}_2F_1(1,1;m+2;x)$. An independent solution has leading
order $x^{-m-1}$ and violates the first condition, since
$p+1<2(m+1)$. Hence
\[
\Phi_\ell(x)=a_\ell\,{}_2F_1(1,1;m+2;x).
\]
It follows that
\[
U_c(w)=Y(w)\,{}_2F_1(1,1;m+2;\|w\|^2)
\]
for a harmonic polynomial $Y$ of bidegree $(p,1)$.

Lemma~\ref{551}, with $\lambda=m+1$, $\mu=m$, and
$\rho=1-x$, gives
\[
{}_2F_1(1,1;m+2;x)
=A(1-x)+(-1)^{m+1}(m+1)x^{-m-1}
(1-x)^m\log(1-x),
\]
where $A$ is analytic near $0$. If $Y\ne0$, choose
$\zeta\in\partial\mathbb B^{b+1}$ with $Y(\zeta)\ne0$.
Along $w=t\zeta$, the factor $Y(t\zeta)=t^{p+1}Y(\zeta)$
is smooth and nonzero near $t=1$, whereas
$(1-t^2)^m\log(1-t^2)$ is not $C^m$ there. This contradicts
the $C^m$ regularity of $U_c$. Hence $U_c=0$. Its boundary value is the
$H^{p,1}(\mathbb C^{b+1})$ component of $g_c$,
so this component vanishes.
\end{proof}

We now convert the vanishing in Lemma~\ref{81} into the boundary
orthogonality needed at the origin.

\begin{proposition}\label{84}
Let $q\ge1$ and $f\in C(S)$.  Suppose that, for every tripotent $c$ of
rank $r-1$, \(\Pi_{H^{q,1}(V_0(c))} \bigl[w\mapsto f(c+w)\bigr]=0\).
Then \(f\perp\mathcal M_q\).
\end{proposition}

\begin{proof}
If $r=1$, then $\Omega$ is the unit ball of $V$,
and $0$ is the only tripotent of rank $r-1$. By \eqref{10},
$\mathcal P_{(q)}$ is the full space of holomorphic homogeneous
polynomials of degree $q$. Thus the left-hand side of \eqref{60}
is the restriction to $S$ of all polynomials of bidegree $(q,1)$.
The decomposition of bihomogeneous polynomials in
\cite[Section~2.2, Lemma~3]{BSW}, restricted to the unit sphere, gives
\[
\operatorname{span}\{(p\bar\ell)|_S:
p\in\mathcal P_{(q)},\ \ell\in\mathcal P_{(1)}\}
=H^{q,1}(V)\oplus\mathcal P_{(q-1)}|_S,
\]
with orthogonal summands. Comparing this with \eqref{60} gives
$\mathcal M_q=H^{q,1}(V)$.
The hypothesis is then $\Pi_{H^{q,1}(V)}f=0$, so
$f\perp\mathcal M_q$.

The case $r>1$ is proved in Appendix~\ref{90}.
\end{proof}

\begin{theorem}\label{70}
Let $f\in L^2(S)$ and put $F=P_{\sigma,\frac nr}f$.
If $F\in C^{b+1+\sigma-\frac nr}(\overline\Omega)$, then
\begin{equation}\label{71}
\bar\partial\mathcal D_{\sigma-\frac nr+1}F=0.
\end{equation}
\end{theorem}

\begin{proof}
Put $q=\sigma-\frac nr+1$. By Lemma~\ref{73}, $f$ has the
continuous representative $F|_S$. Lemma~\ref{81} gives
\(\Pi_{H^{q,1}(V_0(c))}[w\mapsto f(c+w)]=0\)
for every tripotent $c$ of rank $r-1$. Proposition~\ref{84}
therefore yields $f\perp\mathcal M_q$, and Lemma~\ref{61} gives
\(\bar\partial\mathcal D_qF(0)=0\).

For $\widetilde g\in\widetilde G$, put
\(F_{\widetilde g}=P_{\sigma,\frac nr}
\pi_{\sigma,\frac nr}(\widetilde g)f\).
By Lemma~\ref{29}, $F_{\widetilde g}$ has the same boundary
regularity as $F$. Applying the preceding argument to the
boundary datum $\pi_{\sigma,\frac nr}(\widetilde g)f$ gives
\(\bar\partial\mathcal D_qF_{\widetilde g}(0)=0\).
Lemma~\ref{35} now transfers these identities at the origin
to every point of $\Omega$. Hence
$\bar\partial\mathcal D_qF=0$ on $\Omega$, which is \eqref{71}.
\end{proof}

\section{Proof of the main theorem and Bergman-harmonic functions}
\label{87}

We first complete the proof of Theorem~\ref{3}.

\begin{proof}[Proof of Theorem~\ref{3}]
In the scalar realization, $F=P_{\sigma,\theta}f$, where
$\sigma=\frac nr-\frac{s+\nu}{2}$ and
$\theta=\frac nr-\frac{s-\nu}{2}$. The first and second rows
follow from Proposition~\ref{54} and Theorem~\ref{56}, respectively.

In the third row, $s=\nu\in-\mathbb N_0$ and
$\mathbb D_{-\nu,0}F=\mathcal D_{1-\nu}F$ in the
Harish--Chandra trivialization by \eqref{36}. Theorems~\ref{67}
and~\ref{70} give $\bar\partial\mathcal D_{1-\nu}F=0$ in
the tube and non-tube cases, respectively. Since the Dolbeault
operator acts componentwise in this trivialization, this is the
equation in the third row.

In the fourth row, set $m=\sigma\in\mathbb N$ and
$q=-\theta\in\mathbb N_0$. By \eqref{36},
\[
\mathbb D_{q,m}F
=h(z,z)^m\mathcal D_{q+1}\bigl(h(z,z)^{-m}F\bigr).
\]
Thus Proposition~\ref{54} gives the fourth row.

In the fifth row, set $L=\sigma-\frac nr\in\mathbb N_0$
and $j=\frac nr-\theta\in\mathbb N$. The condition
$\theta\notin-\mathbb N_0$ allows us to apply
Theorem~\ref{57}. Since \eqref{36} gives
\[
\mathbb D_{L,j}F
=h(z,z)^j\mathcal D_{L+1}\bigl(h(z,z)^{-j}F\bigr),
\]
that theorem yields the fifth row.

The parameter sets in \eqref{52} are exactly the union of
these five rows. The final row therefore follows from
Proposition~\ref{51}.
\end{proof}

\begin{proof}[Proof of Corollary~\ref{4}]
By \eqref{386}--\eqref{387}, $\Delta_\Omega$ is a positive
constant multiple of the Laplace--Beltrami operator of the
Bergman metric. The assumed boundary regularity gives
$F\in C(\overline\Omega)$, so Inoue's Poisson representation
\cite[Theorem, pp.~532--533]{Inoue} applies.

By \eqref{151}, its Poisson--Szeg\H{o} kernel is
\[
\frac{|C(z,\omega)|^2}{C(z,z)}
=
\frac{h(z,z)^{n/r}}{|h(z,\omega)|^{2n/r}}.
\]
Both representations use the normalized $K$-invariant probability
measure $du$ on $S$.
Hence $F=P_{n/r,n/r}(F|_S)$ by \eqref{8}, with
$F|_S\in C(S)\subset L^2(S)$. This is the scalar form of
$P_{0,0}$. The third row of Theorem~\ref{3} and \eqref{36}
therefore give $\bar\partial\mathcal D_1F=0$.

Since $\mathcal P_{(1)}$ consists of linear holomorphic
polynomials, the coordinates of $\mathcal D_1F$  are
$(\partial_1F,\ldots,\partial_nF)$. Thus
$\bar\partial_j\partial_iF=0$ for all $i,j$, so $F$ is
pluriharmonic.
\end{proof}

\appendix
\section{Proof of Proposition~\ref{84}}
\label{90}

Throughout this appendix $\Omega$ is irreducible and
non-tube, $r>1$, and $q\ge1$. Thus $a,b>0$. Let
\[
\mathcal E_q
=
\operatorname{span}
\{p\overline\ell:p\in\mathcal P_{(q)},\ \ell\in\mathcal P_{(1)}\}.
\]
By \eqref{60}, we have the following orthogonal decomposition:
\begin{equation}\label{91}
\mathcal E_q|_S
=
\mathcal P_{(q-1)}|_S\oplus\mathcal M_q.
\end{equation}
By Lemma~\ref{16},
\[
\mathcal E_q
=
\operatorname{span}
\left\{
\ell_\xi^q\overline{\ell_v}:
\xi,v\in V,\ \operatorname{rank}\xi=1
\right\}.
\]
Together with Lemma~\ref{21}, this also gives
\begin{equation}\label{96}
\mathcal M_q
=
\operatorname{span}
\left\{
\left.
\left(
\ell_\xi^q\overline{\ell_v}
-
\frac{q\langle v,\xi\rangle}
{\frac nr+q-1}\ell_\xi^{q-1}
\right)
\right|_S:
\xi,v\in V,\ \operatorname{rank}\xi=1
\right\}.
\end{equation}
Indeed, \(\partial_v\ell_\xi^q = q\langle v,\xi\rangle\ell_\xi^{q-1}\).
Thus all calculations on $\mathcal E_q$ below may be reduced to
$\ell_\xi^q\overline{\ell_v}$.

 Let $$\mathcal B_{q,1}=\operatorname{span}\{p\overline{\ell}:p\in \bigoplus_{|\alpha|=q}\mathcal P_\alpha,\ \ell\in \mathcal P_{(1)}\}.$$ Equip $\mathcal B_{q,1}$ with the inner product determined by
\begin{equation}\label{91a}
\langle p_1\overline{\ell_1},p_2\overline{\ell_2}\rangle_F
=
\langle p_1,p_2\rangle_F
\langle\ell_2,\ell_1\rangle_F,
\end{equation}
and give $\mathcal E_q\subset\mathcal B_{q,1}$ the induced
inner product.

Fix a tripotent $c_0$ of rank $r-1$.  Recall that \(\Omega_0(c_0) = \Omega\cap V_0(c_0)\)
is the unit ball of the rank-one space $V_0(c_0)$, which has complex
dimension $b+1$, and \(c_0+\partial\Omega_0(c_0)\subset S\).
All Haar measures and sphere measures below have total mass one.

\begin{definition}For $s\in S$, let $\mathcal I_s$ be the set of minimal tripotents
$e\in V$ for which $s-e$ is a rank-$(r-1)$ tripotent orthogonal to $e$.
\end{definition}

Let \(K_s = \{k\in K:ks=s\}\).
For
$e,e'\in\mathcal I_s$, choose Jordan frames such that \[s=e_1+\cdots+e_{r-1}+e, \quad s=e'_1+\cdots+e'_{r-1}+e'.\]
By the transitivity of $K$ on Jordan frames
\cite[Section~2]{KZ}, there is $k\in K$ such that
\(ke_j=e'_j\) for $1\le j\le r-1$ and \(ke=e'\).
Then $ks=s$, so $k\in K_s$ and thereforebv.n[ $K_s$ acts transitively on $\mathcal I_s$.  Fix any $e_0\in\mathcal I_s$.
We define a probability measure $\mu_s$ on $\mathcal I_s$ by
\begin{equation}\label{92}
\int_{\mathcal I_s}f(e)\,d\mu_s(e)
=
\int_{K_s}f(he_0)\,dh,
\qquad f\in C(\mathcal I_s),
\end{equation}
where $dh$ is the normalized Haar measure on $K_s$.
Transitivity of the $K_s$-action and right invariance of Haar measure
show that this definition is independent of $e_0$.
Moreover, the measure $\mu_s$ is $K_s$-invariant.

Fix $e_0\in\partial\Omega_0(c_0)$ and put \(s_0=c_0+e_0\).
Let $d\omega$ be normalized sphere measure on
$\partial\Omega_0(c_0)$.

\begin{lemma}\label{93}
Put \(X = \{(s,e):s\in S,\ e\in\mathcal I_s\}\).
Then, for every $F\in C(X)$,
\begin{equation}\label{94}
\int_K
F(ks_0,ke_0)\,dk=\int_K\int_{\partial\Omega_0(c_0)}
F\bigl(k(c_0+\omega),k\omega\bigr)\,d\omega\,dk
=
\int_S\int_{\mathcal I_s}
F(s,e)\,d\mu_s(e)\,du(s).
\end{equation}
\end{lemma}

\begin{proof}
Since $K$ preserves the rank and orthogonality of tripotents, therefore for $k\in K$ and $(s,e)\in X$, \(k\cdot(s,e):=(ks,ke)\in X\) and then the group $K$ acts on $X$.
Given $(s,e)\in X$,
choose $k_1\in K$ with $k_1s_0=s$. Then $k_1e_0\in\mathcal I_s$.
Since $K_s$ acts transitively on $\mathcal I_s$, there is
$k_2\in K_s$ such that $k_2k_1e_0=e$. Hence
$(s,e)=(k_2k_1s_0,k_2k_1e_0)$ and therefore the action of $K$ on $X$ is transitive. 

For each $\omega\in\partial\Omega_0(c_0)$, the pair
$(c_0+\omega,\omega)$ belongs to $X$. Since $K$ acts transitively
on $X$, right invariance of Haar measure gives
\[
\int_K F\bigl(k(c_0+\omega),k\omega\bigr)\,dk
=
\int_K F(ks_0,ke_0)\,dk.
\]
Integrating over $\omega$ proves the equality of the first two
integrals in \eqref{94}.

By \eqref{92}, $\mu_{ks_0}=k_*\mu_{s_0}$ for $k\in K$.
Also, $du$ is the image of normalized Haar measure under
$k\mapsto ks_0$. Hence
\[
\begin{aligned}
\int_S\int_{\mathcal I_s}F(s,e)\,d\mu_s(e)\,du(s)
&=\int_K\int_{K_{s_0}}F(ks_0,khe_0)\,dh\,dk\\
&=\int_{K_{s_0}}\int_K F(khs_0,khe_0)\,dk\,dh=\int_K F(ks_0,ke_0)\,dk.
\end{aligned}
\]
The second equality uses $hs_0=s_0$; the last uses right
invariance of Haar measure on $K$. This proves \eqref{94}.
\end{proof}

\begin{definition}For a tripotent $c$ of rank $r-1$, define \(R_c:C(S)\longrightarrow H^{q,1}(V_0(c))\)
by \[R_ch = \Pi_{H^{q,1}(V_0(c))} \bigl[w\mapsto h(c+w)\bigr],\ \ h\in C(S).\]
For $P\in\mathcal E_q$, define
\begin{equation}\label{95}
(\mathcal A_qP)(s)
=
\int_{\mathcal I_s}
(R_{s-e}(P|_S))(e)\,d\mu_s(e),
\qquad
s\in S.
\end{equation}
\end{definition}

For $s\in S$, let \(V=V_2(s)\oplus V_1(s)\oplus V_0(s)\) be the corresponding Peirce decomposition. Since  $s$ is a maximal tripotent, therefore  $V_0(s)=\{0\}$. For $j=1,2$, let
\(\operatorname{pr}_j(s):V\longrightarrow V_j(s)\)
be the orthogonal projection. 

\begin{lemma}\label{97}
Let $\xi,v\in V$ with $\operatorname{rank}\xi=1$, let $s\in S$, and
$e\in\mathcal I_s$. Put \[c=s-e, \quad W=V_0(c),\ \Omega_0(c)=W\cap \Omega\]
and let
$
\operatorname{pr}_W:V\longrightarrow W
$
be the orthogonal projection. Then
\begin{equation}\label{98}
R_c\bigl(\ell_\xi^q\overline{\ell_v}|_S\bigr)(w)
=
\ell_\xi(w)^q\overline{\ell_v(w)}
-
\frac1{b+q}
\partial_{\operatorname{pr}_Wv}\ell_\xi^q(w),
\qquad
w\in\partial\Omega_0(c).
\end{equation}
In particular,
\begin{equation}\label{99}
R_c\bigl(\ell_\xi^q\overline{\ell_v}|_S\bigr)(e)
=
\frac b{b+q}
\ell_\xi(e)^q\overline{\ell_v(e)}
-
\frac1{b+q}
\partial_{D(e,\bar s)\operatorname{pr}_1(s)v}
\ell_\xi^q(e).
\end{equation}
\end{lemma}

\begin{proof}
Since $\operatorname{rank}c=r-1$, therefore $\dim_{\mathbb C}W=b+1$. Choose orthonormal
coordinates on $W$ and write
$\Delta_W=\sum_{j=1}^{b+1}\partial_{w_j}\bar\partial_{w_j}$.
For a holomorphic homogeneous polynomial $p$ of degree $q$
and $u\in W$, direct differentiation gives
\[
\Delta_W(p\overline{\ell_u})=\partial_up,
\qquad
\Delta_W(\|w\|^2\partial_up)=(b+q)\partial_up.
\]
These identities show that
\(p\overline{\ell_u}
-\frac{\|w\|^2}{b+q}\partial_up\)
is harmonic and of bidegree $(q,1)$. Its restriction to
$\partial\Omega_0(c)$ therefore lies in $H^{q,1}(W)$.
Since  $\|w\|=1$ on $\partial \Omega_0(c)$ and 
$\partial_up|_{\partial \Omega_0(c)}\in H^{q-1,0}(W)$, and $ H^{q-1,0}(W)$ is orthogonal to
$H^{q,1}(W)$, it follows that
\begin{equation}\label{100}
\Pi_{H^{q,1}(W)}\bigl(p\overline{\ell_u}\bigr)
=
p\overline{\ell_u}|_S
-\frac1{b+q}\partial_up|_S
\qquad\text{on }\partial\Omega_0(c).
\end{equation}

Now consider \[w\longmapsto \ell_\xi(c+w)^q\overline{\ell_v(c+w)}, \quad w\in W.\]
Its term of bidegree $(q,1)$ is
$
\ell_\xi(w)^q\overline{\ell_v(w)}.
$
All the other terms have total degree at most $q$, their spherical
harmonic components therefore have degree at most $q$ and are
orthogonal to $H^{q,1}(W)$. Then
\[\Pi_{H^{q,1}(W)}\bigl[w\longmapsto \ell_\xi(c+w)^q\overline{\ell_v(c+w)}]=\Pi_{H^{q,1}(W)}(\ell_\xi^q\overline{\ell_v}).\]
Apply \eqref{100} to \(p(w)=\ell_\xi(w)^q\) and \( u=\operatorname{pr}_Wv\) and $\ell_v(w)=\ell_{\operatorname{pr}_Wv}(w)$ gives \eqref{98}.

It remains to evaluate \eqref{98} at $e$. Complete $e$ to a Jordan
frame $e_1,\ldots,e_r$ such that \[e_r=e, \quad c=e_1+\cdots+e_{r-1}, \quad s=e_1+\cdots+e_r.\]
Applying \eqref{503} to $c$ and $s$, with $V_{ij}$ as in
\eqref{502}, gives
\[W=V_{rr}\oplus V_{0r}=\mathbb Ce\oplus V_{0r}\ \mbox{and}\
V_1(s)=\bigoplus_{j=1}^rV_{0j}.\]
Since $D(e,\bar s)$ is the identity on $V_{0r}$ and is zero on
$V_{0j}$ for $j<r$, one has \[\operatorname{pr}_Wv = \overline{\ell_v(e)}\,e + D(e,\bar s)\operatorname{pr}_1(s)v.\]
Also, \(\partial_e\ell_\xi^q(e) = q\ell_\xi(e)^q\).
Therefore
\[
\partial_{\operatorname{pr}_Wv}\ell_\xi^q(e)
=
q\ell_\xi(e)^q\overline{\ell_v(e)}
+
\partial_{D(e,\bar s)\operatorname{pr}_1(s)v}
\ell_\xi^q(e).
\]
Substituting this into \eqref{98} at $w=e$ gives \eqref{99}.
\end{proof}

\begin{lemma}\label{1000}
For $m\in\mathbb N_0$ and $y\in V$, the Fischer kernels satisfy
\begin{equation}\label{1001}
\left(\frac a2+m\right)
\bar\partial_y^wK^F_{(m+1)}(z,w)
=
\frac a2\ell_y(z)K^F_{(m)}(z,w)
+
\partial_{Q(z)\bar y}^zK^F_{(m)}(z,w).
\end{equation}
If $\xi\in V$ has rank one, then
\begin{equation}\label{1002}
\frac a2\ell_y(z)\ell_\xi(z)^m
+
\partial_{Q(z)\bar y}\ell_\xi^m(z)
=
m!\left(\frac a2+m\right)
\bar\partial_y^wK^F_{(m+1)}(z,w)\big|_{w=\xi}.
\end{equation}
In particular, the polynomial on the left of \eqref{1002} belongs to
$\mathcal P_{(m+1)}$.
\end{lemma}

\begin{proof}
For $z,w$ near $0$, recall that
\(z^{\bar w}=B(z,\bar w)^{-1}(z-Q(z)\bar w)\) is the quasi-inverse.
The symmetry principle gives \(z^{\bar w}=z+Q(z)\overline{w^{\bar z}}\).
The logarithmic derivative formulas for the generic norm
\cite[formulas~(1.13)--(1.14)]{WYZR} give
\[
\bar\partial_y^w h(z,w)^{-a/2}
=
\frac a2h(z,w)^{-a/2}\langle z^{\bar w},y\rangle,
\qquad
\partial_u^z h(z,w)^{-a/2}
=
\frac a2h(z,w)^{-a/2}\langle u,w^{\bar z}\rangle.
\]
By the outer symmetry of the Jordan triple product and the adjoint
relation for the trace form
\cite[Section~2.4, Theorem~2.2 (JT2) and formulas~(7)--(8)]{Clerc},
\begin{equation}\label{106}
\begin{aligned}
\langle Q(z)\bar u,y\rangle
&=
\frac12\langle D(z,\bar u)z,y\rangle
=
\frac12\langle z,D(u,\bar z)y\rangle
=
\frac12\langle z,D(y,\bar z)u\rangle
\\
&=
\frac12\langle D(z,\bar y)z,u\rangle
=
\langle Q(z)\bar y,u\rangle.
\end{aligned}
\end{equation}
Hence
\[
\bar\partial_y^w h(z,w)^{-a/2}
=
\frac a2\ell_y(z)h(z,w)^{-a/2}
+
\partial_{Q(z)\bar y}^z h(z,w)^{-a/2}.
\]
At the parameter $a/2$, \eqref{13} becomes
\[h(z,w)^{-a/2}=\sum_{j\ge0}(a/2)_jK^F_{(j)}(z,w).\]
Comparing the terms of degree $m+1$ in $z$ and degree $m$ in
$\bar w$ proves \eqref{1001} near $(0,0)$, and hence everywhere
because both sides are polynomials. Setting $w=\xi$ and using
Lemma~\ref{16} (and $K^F_{(0)}=1$ when $m=0$) gives \eqref{1002}.
As a function of $z$, the right-hand side of \eqref{1002} lies in
$\mathcal P_{(m+1)}$.
\end{proof}

\begin{lemma}\label{101}
Let $m\in\mathbb N_0$, $p\in\mathcal P_{(m)}$, $s\in S$, $v\in V$,
and $x\in V_1(s)$. Then
\begin{equation}\label{102}
\int_{\mathcal I_s}p(e)\,d\mu_s(e)
=
\frac{(a/2)_m}{(ar/2)_m}p(s).
\end{equation}
\begin{equation}\label{103}
\int_{\mathcal I_s}
p(e)\overline{\ell_v(e)}\,d\mu_s(e)
=
\frac{(a/2)_m}{(ar/2)_{m+1}}
\left[
\frac a2p(s)\overline{\ell_v(s)}
+
\partial_{\operatorname{pr}_2(s)v}p(s)
\right].
\end{equation}
\begin{equation}\label{104}
\int_{\mathcal I_s}
\partial_{D(e,\bar s)x}p(e)\,d\mu_s(e)
=
\frac{(a/2)_m}{(ar/2)_m}\partial_xp(s).
\end{equation}
\end{lemma}

\begin{proof}
For $m=0$, \eqref{102} follows from the normalization of $\mu_s$.
To prove \eqref{102} for $m\ge1$, fix $e_0\in\mathcal I_s$. By \eqref{92}, one has, for $k\in K$,
\[
\begin{aligned}
\int_{\mathcal I_s}(k\cdot p)(e)\,d\mu_s(e)
=
\int_{K_s}p(k^{-1}he_0)\,dh
=
\int_{K_{k^{-1}s}}
p(hk^{-1}e_0)\,dh
=
\int_{\mathcal I_{k^{-1}s}}
p(e)\,d\mu_{k^{-1}s}(e).
\end{aligned}
\]
Here we used \(K_{k^{-1}s}=k^{-1}K_sk\)
and the invariance of normalized Haar measure under conjugation.

By \eqref{12} and Frobenius reciprocity, the $K$-type
$\mathcal P_{(m)}$ occurs with multiplicity one in $L^2(S)$.
The restriction map
$
p\longmapsto p|_S
$
is a nonzero $K$-equivariant map from $\mathcal P_{(m)}$ into
$L^2(S)$. We have just shown that
$
p\longmapsto
\left[
s\longmapsto
\int_{\mathcal I_s}p(e)\,d\mu_s(e)
\right]
$
is another $K$-equivariant map from $\mathcal P_{(m)}$ into
$L^2(S)$. By multiplicity one, there is a constant $c_m$ such that
\begin{equation}\label{105}
\int_{\mathcal I_s}p(e)\,d\mu_s(e)
=
c_mp(s),
\qquad
p\in\mathcal P_{(m)},\ s\in S.
\end{equation}

Take \(p=K^F_{(m)}(\,\cdot\,,s)\).
If $e\in\mathcal I_s$, then $s-e$ is orthogonal to $e$, and hence
$
\langle e,s\rangle=1.
$
By the Hermitian symmetry of $K^F_{(m)}$ and Lemma~\ref{16}, one has
 \[\int_{\mathcal I_s} K^F_{(m)}(e,s)\,d\mu_s(e) =\int_{\mathcal I_s} \overline{K^F_{(m)}(s,e)}\,d\mu_s(e)=\int_{\mathcal I_s} \frac1{m!}\,d\mu_s(e)= \frac1{m!}.\]

On the other hand, for $0<t<1$, \(h(\lambda s,ts)^{-a/2} = (1-\lambda t)^{-ar/2}\).
Since only one-row terms occur in \eqref{13} at the parameter $a/2$,
comparison of the coefficients of $(\lambda t)^m$ gives \[K^F_{(m)}(s,s) = \frac{(ar/2)_m}{(a/2)_m\,m!}.\]
Substituting these two values into \eqref{105} gives \(c_m = \frac{(a/2)_m}{(ar/2)_m}\).
This proves \eqref{102}.

We next prove \eqref{103}. For $m\ge1$, Lemma~\ref{16} reduces
this to $p=\ell_\xi^m$ with $\xi$ of rank one. For $m=0$, it suffices
to take $p=1=\ell_\xi^0$. Put $y=Q(s)\bar v$. By Lemma~\ref{1000},
the polynomial on the left of \eqref{1002} belongs to
$\mathcal P_{(m+1)}$.

Now let $e\in\mathcal I_s$.
For a maximal tripotent $s$, the Peirce rules for $Q(s)$ give \[Q(s)\overline{Q(s)\bar v}=\operatorname{pr}_2(s)v.\]
Moreover $e\in V_2(s)$ and $Q(s)\bar e=e$.  Using the adjoint relation
for $Q(s)$, or equivalently \eqref{106}, we therefore obtain
\[
\ell_y(e)=\overline{\ell_v(e)},
\qquad
\ell_y(s)=\overline{\ell_v(s)},
\qquad
Q(s)\bar y=\operatorname{pr}_2(s)v.
\]
See also \cite[Section~5.3]{Clerc}.
Since $e$ is minimal, one has $$Q(e)\bar{y}\in V_2(e) ,\  V_2(e)=\mathbb{C}e \ \mbox{and}\ Q(e)\bar{e}=e,$$ hence by  \eqref{106},  \(Q(e)\bar y=\ell_y(e)e\).
Therefore
\[
\begin{aligned}
\partial_{Q(e)\bar y}\ell_\xi^m(e)
&=
\ell_y(e)\partial_e\ell_\xi^m(e)=
m\ell_\xi(e)^m\overline{\ell_v(e)}.
\end{aligned}
\]
The polynomial on the left of \eqref{1002} therefore takes the value
\(\left(\frac a2+m\right)\ell_\xi(e)^m\overline{\ell_v(e)}\)
at $e$. At $s$, its value is
\[
\frac a2
\ell_\xi(s)^m\overline{\ell_v(s)}
+
\partial_{\operatorname{pr}_2(s)v}\ell_\xi^m(s).
\]
Applying \eqref{102} in degree $m+1$ to this polynomial gives
\[
\begin{aligned}
\left(\frac a2+m\right)
\int_{\mathcal I_s}
\ell_\xi(e)^m\overline{\ell_v(e)}\,d\mu_s(e)
={}&
\frac{(a/2)_{m+1}}{(ar/2)_{m+1}}
\left[
\frac a2
\ell_\xi(s)^m\overline{\ell_v(s)}
+
\partial_{\operatorname{pr}_2(s)v}\ell_\xi^m(s)
\right].
\end{aligned}
\]
Dividing by $a/2+m$ gives
\[
\int_{\mathcal I_s}
\ell_\xi(e)^m\overline{\ell_v(e)}\,d\mu_s(e)
=
\frac{(a/2)_m}{(ar/2)_{m+1}}
\left[
\frac a2
\ell_\xi(s)^m\overline{\ell_v(s)}
+
\partial_{\operatorname{pr}_2(s)v}\ell_\xi^m(s)
\right].
\]
For $m\ge1$, the rank-one powers $\ell_\xi^m$ span
$\mathcal P_{(m)}$ by Lemma~\ref{16}; for $m=0$, they span the constants.
Hence \eqref{103} follows by linearity.

Finally, the Harish--Chandra decomposition
\(\mathfrak g_{\mathbb C}=V\oplus\mathfrak k_{\mathbb C}\oplus\bar V\)
and the bracket relation \([x,\bar y]=D(x,\bar y)\) in
\cite[Section~2, formulas~(1)--(2)]{KZ} show that
\(D(x,\bar s)\in\mathfrak k_{\mathbb C}\).
Since $\mathcal P_{(m)}$ is $K$-invariant, it is also invariant under
the derived action of $\mathfrak k_{\mathbb C}$. More explicitly, if $A\in\mathfrak k$, then \(p(e^{tA}z)\in\mathcal P_{(m)}\)
for every $t$, and differentiation at $t=0$ gives \(z\longmapsto\partial_{Az}p(z)\in\mathcal P_{(m)}\).
By complex linearity, the same conclusion holds for
$A\in\mathfrak k_{\mathbb C}$. Hence \(z\longmapsto \partial_{D(x,\bar s)z}p(z) \in\mathcal P_{(m)}\).

By the outer symmetry of the Jordan triple product, \[D(x,\bar s)s = D(s,\bar s)x = x.\]
We may therefore apply \eqref{102} to the polynomial \(z\longmapsto \partial_{D(x,\bar s)z}p(z)\).
This gives
\[
\begin{aligned}
\int_{\mathcal I_s}
\partial_{D(e,\bar s)x}p(e)\,d\mu_s(e)
&=
\int_{\mathcal I_s}
\partial_{D(x,\bar s)e}p(e)\,d\mu_s(e)=
\frac{(a/2)_m}{(ar/2)_m}
\partial_{D(x,\bar s)s}p(s)=
\frac{(a/2)_m}{(ar/2)_m}
\partial_xp(s).
\end{aligned}
\]
This proves \eqref{104}.
\end{proof}

The preceding averaging formulas lead to the following operator.
Fix an orthonormal basis $v_1,\ldots,v_n$ of $V$ and write \(z_j=\langle z,v_j\rangle\).
Since $\bar z_1,\ldots,\bar z_n$ form a basis of
$\overline{\mathcal P_{(1)}}$, every $U\in\mathcal E_q$ has a unique
expansion \(U(z)=\sum_{j=1}^np_j(z)\bar z_j,\) where \( p_j\in\mathcal P_{(q)}\).
Define
\[
(\mathsf TU)(z)
=
\frac{ab}{2}U(z)
+
\left(b+\frac{ar}{2}+q\right)
\sum_{j=1}^n
\partial_{D(z,\bar z)v_j}p_j(z).
\]

\begin{lemma}\label{107}
The operator $\mathsf T$ maps $\mathcal E_q$ into itself and is
invertible.
\end{lemma}

\begin{proof}
To see that $\mathsf T$ maps $\mathcal E_q$ into itself, let \[U(z)=\sum_{j=1}^np_j(z)\bar z_j, \quad p_j\in\mathcal P_{(q)}.\]
Since
$
\bar z=\sum_{l=1}^n\bar z_l\bar v_l,
$
the outer symmetry gives
\[
D(z,\bar z)v_j
=
\sum_{l=1}^n
\bar z_lD(z,\bar v_l)v_j
=
\sum_{l=1}^n
\bar z_lD(v_j,\bar v_l)z.
\]
Hence
\[
\sum_{j=1}^n
\partial_{D(z,\bar z)v_j}p_j(z)
=
\sum_{l=1}^n
\left(
\sum_{j=1}^n
\partial_{D(v_j,\bar v_l)z}p_j(z)
\right)\bar z_l.
\]
The proof of Lemma~\ref{101} shows that \(z\longmapsto \partial_{D(v_j,\bar v_l)z}p_j(z)\)
belongs to $\mathcal P_{(q)}$. Therefore the coefficient of each
$\bar z_l$ in the preceding expression belongs to
$\mathcal P_{(q)}$, and hence \[\sum_{j=1}^n \partial_{D(z,\bar z)v_j}p_j(z) \in\mathcal E_q.\]
Thus $\mathsf T\mathcal E_q\subset\mathcal E_q$.

It remains to prove invertibility.  Since $K$ acts unitarily on $V$,
$\mathfrak k$ consists of skew-Hermitian endomorphisms.  Moreover,
$D(z,\bar z)\in\mathfrak k_{\mathbb C}$ and
$D(z,\bar z)^*=D(z,\bar z)$.  Writing
$\mathfrak k_{\mathbb C}=\mathfrak k\oplus i\mathfrak k$ as real vector
spaces, the Hermitian condition therefore gives \(D(z,\bar z)\in i\mathfrak k\).

Let $\mathcal B$ be the positive multiple of the complexified Killing form
for which \[\mathcal B(u,\bar v)=\langle u,v\rangle, \quad u,v\in V.\]
Its restriction to the real space $i\mathfrak k$ is positive definite.
Choose a $\mathcal B$-orthonormal basis $(Z_\gamma)$ of $i\mathfrak k$.
Then every $Z_\gamma$ acts Hermitianly on $V$.  Using \([z,\bar w]=D(z,\bar w)\)
and the invariance of $\mathcal B$, we obtain
\[
\begin{aligned}
\mathcal B(D(z,\bar z),Z_\gamma)
&=
\mathcal B([z,\bar z],Z_\gamma)
=
\mathcal B([Z_\gamma,z],\bar z)
=
\langle Z_\gamma z,z\rangle.
\end{aligned}
\]
Therefore
\begin{equation}\label{107a}
D(z,\bar z)
=
\sum_\gamma
\langle Z_\gamma z,z\rangle Z_\gamma.
\end{equation}

For fixed $\gamma$, define
\[
\mathsf C_\gamma U
=
\sum_{j=1}^n
\partial_{Z_\gamma v_j}p_j,
\qquad
U=\sum_{j=1}^np_j\bar z_j.
\]
By Lemma~\ref{16}, \(\mathsf C_\gamma U\in\mathcal P_{(q-1)}\).
If \(Z_\gamma v_j = \sum_l c^\gamma_{lj}v_l\),
then, because $Z_\gamma$ is Hermitian, \(c^\gamma_{jl}=\overline{c^\gamma_{lj}}\).
For $h\in\mathcal P_{(q-1)}$, Fischer adjointness in the two variables gives
\begin{equation}\label{107b}
\left\langle
\langle Z_\gamma z,z\rangle h,
U
\right\rangle_F
=
\langle h,\mathsf C_\gamma U\rangle_F.
\end{equation}
Indeed, \(\langle Z_\gamma z,z\rangle = \sum_{l,j}c^\gamma_{lj}z_j\bar z_l\),
so the left-hand side of \eqref{107b} is \(\sum_{l,j}c^\gamma_{lj} \langle h,\partial_jp_l\rangle_F\).
On the other hand, after interchanging $j$ and $l$,
\[
\mathsf C_\gamma U
=
\sum_{l,j}c^\gamma_{jl}\partial_jp_l
=
\sum_{l,j}\overline{c^\gamma_{lj}}\partial_jp_l.
\]
Since the Fischer inner product is conjugate-linear in its second variable,
this gives exactly \eqref{107b}.
Notice that we only use the pairing identity \eqref{107b}; no assertion is
needed that multiplication by $\langle Z_\gamma z,z\rangle$ preserves
$\mathcal E_q$.

By \eqref{107a},
\[
\sum_{j=1}^n
\partial_{D(z,\bar z)v_j}p_j
=
\sum_\gamma
\langle Z_\gamma z,z\rangle
\mathsf C_\gamma U.
\]
Taking the Fischer inner product with $U$ and applying \eqref{107b} with
$h=\mathsf C_\gamma U$ gives
\[
\left\langle
\sum_{j=1}^n
\partial_{D(z,\bar z)v_j}p_j,
U
\right\rangle_F
=
\sum_\gamma
\|\mathsf C_\gamma U\|_F^2.
\]
Consequently
\begin{equation}\label{108}
\langle\mathsf TU,U\rangle_F
=
\frac{ab}{2}\|U\|_F^2
+
\left(b+\frac{ar}{2}+q\right)
\sum_\gamma
\|\mathsf C_\gamma U\|_F^2.
\end{equation}
Since $a>0$ and $b>0$, the right-hand side is positive whenever
$U\ne0$.  Thus $\mathsf T$ is injective.  Since $\mathcal E_q$ is
finite-dimensional, $\mathsf T$ is invertible.
\end{proof}

\begin{lemma}\label{109}
For $\mathcal A_q$ defined in \eqref{95}, \(\operatorname{Ran}\mathcal A_q = \mathcal M_q\).
\end{lemma}

\begin{proof}
We first consider \(U = \ell_\xi^q\overline{\ell_v}\), where
\(\xi,v\in V\) and \(\operatorname{rank}\xi=1\).
By \eqref{99}, \eqref{103}, and \eqref{104}, with \(x=\operatorname{pr}_1(s)v\),
we have
\[
\begin{aligned}
(\mathcal A_qU)(s)
=
\frac{(a/2)_q}{(b+q)(ar/2)_{q+1}}
\Bigg\{
b\left[
\frac a2\ell_\xi(s)^q\overline{\ell_v(s)}
+
\partial_{\operatorname{pr}_2(s)v}\ell_\xi^q(s)
\right]
-
\left(\frac{ar}{2}+q\right)
\partial_{\operatorname{pr}_1(s)v}\ell_\xi^q(s)
\Bigg\}.
\end{aligned}
\]
Since $s$ is maximal, \[V = V_2(s)\oplus V_1(s),\ 
\mbox{and}\ D(s,\bar s) = 2\operatorname{pr}_2(s) + \operatorname{pr}_1(s).\]
Therefore
\[
\operatorname{pr}_2(s)
=
D(s,\bar s)-I,
\qquad
\operatorname{pr}_1(s)
=
2I-D(s,\bar s),
\]
and hence
\begin{align}
(\mathcal A_qU)(s)
={}&
\frac{(a/2)_q}{(b+q)(ar/2)_{q+1}}
\Bigg[
\frac{ab}{2}
\ell_\xi(s)^q\overline{\ell_v(s)}
+
\left(b+\frac{ar}{2}+q\right)
\partial_{D(s,\bar s)v}\ell_\xi^q(s)
\notag\\
&\hspace{40mm}
-
(b+ar+2q)\partial_v\ell_\xi^q(s)
\Bigg].
\label{110}
\end{align}

Write \(v=\sum_{j=1}^n a_jv_j, \quad a_j=\langle v,v_j\rangle\).
Then \(\overline{\ell_v(z)} = \sum_{j=1}^n a_j\bar z_j\).
By the definition of $\mathsf T$, the first two terms in the brackets
in \eqref{110} are $(\mathsf TU)(s)$. Also, \[\partial_v\ell_\xi^q = q\langle v,\xi\rangle\ell_\xi^{q-1} \in \mathcal P_{(q-1)}.\]
Thus
\[
\mathcal A_qU
-
\frac{(a/2)_q}{(b+q)(ar/2)_{q+1}}
(\mathsf TU)|_S
\in
\mathcal P_{(q-1)}|_S.
\]

By Lemma~\ref{16},
$\mathcal P_{(q)}$ is spanned by the polynomials
$\ell_\xi^q$ with $\operatorname{rank}\xi=1$. Hence
\[
\mathcal E_q
=
\operatorname{span}
\left\{
\ell_\xi^q\overline{\ell_v}:
\xi,v\in V,\ \operatorname{rank}\xi=1
\right\}.
\]
It follows by linearity that
\begin{equation}\label{111}
\mathcal A_qU
-
\frac{(a/2)_q}{(b+q)(ar/2)_{q+1}}
(\mathsf TU)|_S
\in
\mathcal P_{(q-1)}|_S,
\qquad
U\in\mathcal E_q.
\end{equation}

We next record a formula for the pairing with $\mathcal A_qU$.
For $g\in C(S)$ and $U\in\mathcal E_q$,
\begin{equation}\label{112}
\langle g,\mathcal A_qU\rangle_S
=
\int_K
\left\langle
R_{kc_0}g,
R_{kc_0}(U|_S)
\right\rangle
\,dk.
\end{equation}
To prove this, first suppose that \(U=\ell_\xi^q\overline{\ell_v}\) with\( \operatorname{rank}\xi=1\).
By \eqref{99} and \(\operatorname{pr}_1(s)=2I-D(s,\bar s)\),
the function \((s,e) \longmapsto (R_{s-e}(U|_S))(e)\)
is continuous on $X$. The same is therefore true for every
$U\in\mathcal E_q$ by linearity. Hence \eqref{95} and \eqref{94} give
\[
\begin{aligned}
\langle g,\mathcal A_qU\rangle_S
=
\int_K\int_{\partial\Omega_0(c_0)}
&g\bigl(k(c_0+\omega)\bigr)
\overline{(R_{kc_0}(U|_S))(k\omega)}
\,d\omega\,dk.
\end{aligned}
\]
For fixed $k\in K$, the map \(\omega\longmapsto k\omega\)
carries the normalized sphere measure on
$\partial\Omega_0(c_0)$ onto that on
$\partial\Omega_0(kc_0)$. Thus the inner integral is
\[
\int_{\partial\Omega_0(kc_0)}
g(kc_0+e)
\overline{(R_{kc_0}(U|_S))(e)}
\,de.
\]
Since $R_{kc_0}$ is the orthogonal projection onto
$H^{q,1}(V_0(kc_0))$, this is \(\left\langle R_{kc_0}g, R_{kc_0}(U|_S) \right\rangle\).
This proves \eqref{112}.

We now show that
\begin{equation}\label{113}
\mathcal A_qU
\perp
\mathcal P_{(q-1)}|_S,
\qquad
U\in\mathcal E_q.
\end{equation}
Let $h\in\mathcal P_{(q-1)}$. For every tripotent $c$ of rank $r-1$,
the function \(e\longmapsto h(c+e)\)
is a holomorphic polynomial of degree at most $q-1$ on $V_0(c)$.
Its homogeneous components belong to \(H^{m,0}(V_0(c)), \quad 0\le m\le q-1\),
and are orthogonal on the sphere to
$H^{q,1}(V_0(c))$. Hence \(R_ch=0\).
Applying \eqref{112} with $g=h$ gives \(\langle h,\mathcal A_qU\rangle_S=0\),
and \eqref{113} follows.

Since $\mathsf TU\in\mathcal E_q$, \eqref{111} and \eqref{91} imply \(\mathcal A_qU\in\mathcal E_q|_S\).
Now \eqref{91} and \eqref{113} give \(\mathcal A_qU\in\mathcal M_q\).

By Lemma~\ref{21}, the restriction of $\Pi_H$ to
$\mathcal E_q|_S$ is the orthogonal projection onto
$\mathcal P_{(q-1)}|_S$. Taking the $\mathcal M_q$ component in
\eqref{111}, we obtain
\begin{equation}\label{114}
\mathcal A_qU
=
\frac{(a/2)_q}{(b+q)(ar/2)_{q+1}}
(I-\Pi_H)\bigl((\mathsf TU)|_S\bigr).
\end{equation}
The scalar in \eqref{114} is nonzero. Since $\mathsf T$ is invertible, \(\mathsf T(\mathcal E_q) = \mathcal E_q\).
Therefore \eqref{91} and \eqref{114} give \(\operatorname{Ran}\mathcal A_q = (I-\Pi_H)(\mathcal E_q|_S) = \mathcal M_q\).
\end{proof}

\begin{proof}[Completion of the proof of Proposition~\ref{84}]
We are in the case $r>1$. The hypothesis says that $R_cf=0$
for every tripotent $c$ of rank $r-1$. In particular,
$R_{kc_0}f=0$ for every $k\in K$. Hence \eqref{112} gives
\(\langle f,\mathcal A_qU\rangle_S=0\) for every
$U\in\mathcal E_q$. Thus
\(f\perp\operatorname{Ran}\mathcal A_q=\mathcal M_q\),
where the equality is Lemma~\ref{109}.
\end{proof}


\begin{thebibliography}{99}                                                                                               %
\raggedright

\bibitem {ABC}P.~Ahern, J.~Bruna, and M.~C.~Cascante, \emph{$H^{p}$-theory for
generalized $M$-harmonic functions in the unit ball}, Indiana Univ. Math. J.
\textbf{45} (1996), no.~1, 103--135.

\bibitem {BBDHJ}A.~Bonami, D.~Buraczewski, E.~Damek, A.~Hulanicki, and
P.~Jaming, \emph{Maximum boundary regularity of bounded Hua-harmonic functions
on tube domains}, J. Geom. Anal. \textbf{14} (2004), no.~3, 457--486.

\bibitem {BerlineVergne}N.~Berline and M.~Vergne,
\emph{\'Equations de Hua et noyau de Poisson}, in \emph{Non-commutative
harmonic analysis and Lie groups}, Lecture Notes in Math., vol.~880,
Springer, Berlin, 1981, 1--51.

\bibitem {Boussejra03}A.~Boussejra,
\emph{$L^2$-Poisson integral representations of solutions of the Hua system
on the bounded symmetric domain $SU(n,n)/S(U(n)\times U(n))$},
J. Funct. Anal. \textbf{202} (2003), no.~1, 25--43.

\bibitem {Boussejra06}A.~Boussejra,
\emph{$L^p$-Poisson integral representations of solutions of the Hua system
on Hermitian symmetric spaces of tube type},
J. Funct. Anal. \textbf{235} (2006), no.~2, 413--429.

\bibitem {BK07}A.~Boussejra and K.~Koufany,
\emph{Characterization of the Poisson integrals for the non-tube bounded
symmetric domains}, J. Math. Pures Appl. (9) \textbf{87} (2007), no.~4,
438--451.

\bibitem {Boussejra11}A.~Boussejra,
\emph{Boundary behavior of Poisson integrals on boundaries of symmetric
spaces}, J. Lie Theory \textbf{21} (2011), no.~2, 243--261.

\bibitem {BO}A.~Boussejra and N.~Ourchane, \emph{$L^{p}$-Poisson integral
representations for solutions of the generalized Hua system on line bundles
over $SU(n,n)/S(U(n)\times U(n))$}, J. Funct. Anal. \textbf{280} (2021),
no.~7, 108824.

\bibitem {ChenLi}R.-Y.~Chen and S.-Y.~Li, \emph{Graham type theorem on
classical bounded symmetric domains}, Calc. Var. PDE \textbf{58} (2019),
article 81.

\bibitem {Clerc}J.-L.~Clerc, \emph{Geometry of the Shilov boundary of a
bounded symmetric domain}, J. Geom. Symmetry Phys. \textbf{13} (2009), 25--74.

\bibitem {BSW}H.~De Bie, F.~Sommen, and M.~Wutzig,
\emph{Plane wave formulas for spherical, complex and symplectic harmonics},
J. Approx. Theory \textbf{222} (2017), 110--131.

\bibitem {EWO}F.~El Wassouli and D.~Oukacha, \emph{Hua operators on
homogeneous line bundles over non-tube type bounded symmetric domains},
Differential Geom. Appl. \textbf{96} (2024), 102168.

\bibitem {EYZ}M.~Engli\v{s}, E.-H.~Youssfi, and G.~Zhang, \emph{Weighted
Bergman kernels for nearly holomorphic functions on bounded symmetric
domains}, J. Funct. Anal. \textbf{286} (2024), no.~2, 110213.
Formula and page references follow
\href{https://arxiv.org/abs/2303.02256v1}{arXiv:2303.02256v1}.

\bibitem {FK}J.~Faraut and A.~Kor\'anyi, \emph{Function spaces and reproducing
kernels on bounded symmetric domains}, J. Funct. Anal. \textbf{88} (1990),
64--89.

\bibitem {FultonHarris}W.~Fulton and J.~Harris,
\emph{Representation Theory: A First Course},
Graduate Texts in Mathematics, vol.~129, Springer-Verlag,
New York, 1991.

\bibitem {Graham}C.~R.~Graham, \emph{The Dirichlet problem for the Bergman
Laplacian. I}, Comm. Partial Differential Equations \textbf{8} (1983), no.~5,
433--476.
\emph{II}, ibid. \textbf{8} (1983), no.~6, 563--641.

\bibitem {GrahamLee}C.~R.~Graham and J.~M.~Lee,
\emph{Smooth solutions of degenerate Laplacians on strictly pseudoconvex
domains}, Duke Math. J. \textbf{57} (1988), no.~3, 697--720.

\bibitem {Helgason}S.~Helgason, \emph{Differential Geometry, Lie Groups,
and Symmetric Spaces}, Pure and Applied Mathematics, vol.~80,
Academic Press, New York, 1978.

\bibitem {HuaBook}L.-K.~Hua, \emph{Harmonic analysis of functions of
several complex variables in the classical domains}, Translations of
Mathematical Monographs, vol.~6, American Mathematical Society, Providence,
RI, 1963.

\bibitem {HuaLook}L.-K.~Hua and K.~H.~Look, \emph{Theory of harmonic functions
in classical domains}, Scientia Sinica \textbf{8} (1959), no.~10, 1031--1094.

\bibitem {Inoue}T.~Inoue, \emph{Integral formulas for harmonic functions
associated with boundaries of a bounded symmetric domain}, Osaka J. Math.
\textbf{26} (1989), no.~3, 527--534.

\bibitem {JohnsonKoranyi}K.~D.~Johnson and A.~Kor\'anyi,
\emph{The Hua operators on bounded symmetric domains of tube type},
Ann. of Math. (2) \textbf{111} (1980), no.~3, 589--608.

\bibitem {Kaup02}W.~Kaup, \emph{Bounded symmetric domains and derived
geometric structures}, Rend. Mat. Acc. Lincei (9) \textbf{13} (2002),
no.~3--4, 243--257.
\url{https://www.math.uni-tuebingen.de/user/kaup/rom.pdf}.

\bibitem {Koranyi65}A.~Kor\'anyi, \emph{The Poisson integral for generalized
half-planes and bounded symmetric domains}, Ann. of Math. (2) \textbf{82}
(1965), 332--350.

\bibitem {KoranyiMalliavin}A.~Kor\'anyi and P.~Malliavin,
\emph{Poisson formula and compound diffusion associated to an overdetermined
elliptic system on the Siegel halfplane of rank two}, Acta Math. \textbf{134}
(1975), 185--209.

\bibitem {Koranyi}A.~Kor\'anyi, \emph{Poisson transforms for line bundles from
the Shilov boundary to bounded symmetric domains}, in \emph{Lie Groups:
Structure, Actions, and Representations}, Progr. Math., vol.~306,
Birkh\"auser/Springer, New York, 2013, 141--162.

\bibitem {KZ06}K.~Koufany and G.~Zhang, \emph{Hua operators and Poisson
transform for bounded symmetric domains}, J. Funct. Anal. \textbf{236} (2006),
no.~2, 546--580.

\bibitem {KZ}K.~Koufany and G.~Zhang, \emph{Hua operators, Poisson transform
and relative discrete series on line bundles over bounded symmetric domains},
J. Funct. Anal. \textbf{262} (2012), no.~9, 4140--4159.

Internal references follow the preprint
\href{https://arxiv.org/abs/1105.3806v3}{arXiv:1105.3806v3}.

\bibitem {Lassalle}M.~Lassalle,
\emph{Les \'equations de Hua d'un domaine born\'e sym\'etrique du type tube},
Invent. Math. \textbf{77} (1984), no.~1, 129--161.

\bibitem {LiNi}S.-Y.~Li and L.~Ni,
\emph{On the holomorphicity of proper harmonic maps between unit balls with
the Bergman metrics}, Math. Ann. \textbf{316} (2000), no.~2, 333--354.

\bibitem {LiSimon}S.-Y.~Li and E.~Simon,
\emph{Boundary behavior of harmonic functions in metrics of Bergman type on
the polydisc}, Amer. J. Math. \textbf{124} (2002), no.~5, 1045--1057.

\bibitem {LiWei}S.-Y.~Li and D.~Wei,
\emph{On the rigidity theorem for harmonic functions in K\"ahler metric of
Bergman type}, Sci. China Math. \textbf{53} (2010), no.~3, 779--790.

\bibitem {MisraUpmeierBoundary}G.~Misra and H.~Upmeier,
\emph{Toeplitz $C^*$-algebras on boundary orbits of symmetric domains},
in \emph{Operator Algebras, Toeplitz Operators and Related Topics},
Oper. Theory Adv. Appl. \textbf{279}, Springer, 2020, 307--341.

\bibitem {DLMF}NIST Digital Library of Mathematical Functions, Chapter 15,
especially 15.8 and 15.4. \url{https://dlmf.nist.gov/15}.

\bibitem {Shimeno96}N.~Shimeno,
\emph{Boundary value problems for the Shilov boundary of a bounded symmetric
domain of tube type}, J. Funct. Anal. \textbf{140} (1996), no.~1, 124--141.

\bibitem {Shimeno}N.~Shimeno,
\emph{Boundary value problems for various boundaries of Hermitian symmetric
spaces}, J. Funct. Anal. \textbf{170} (2000), no.~2, 265--285.

\bibitem {UpmeierKinv}H.~Upmeier, \emph{$K$-invariant Hilbert modules and
singular vector bundles on bounded symmetric domains}, J. Reine Angew. Math.
\textbf{799} (2023), 155--187.

\bibitem {UpmeierToeplitz21}H.~Upmeier,
\emph{Eigenvalues of $K$-invariant Toeplitz operators on bounded symmetric
domains}, Integral Equations Operator Theory \textbf{93} (2021), article~27.

\bibitem {WYZR}
A.~Wang, W.~Yin, L.~Zhang, and G.~Roos,
\emph{The K\"ahler--Einstein metric for some Hartogs domains over symmetric domains},
Sci. China Ser. A \textbf{49} (2006), no.~9, 1175--1210.

\bibitem {Xiao}M.~Xiao,
\emph{Bergman-harmonic functions on classical domains},
Int. Math. Res. Not. IMRN \textbf{2021} (2021), no.~22, 17220--17255.

\bibitem {ZhangBerezin00}G.~Zhang,
\emph{Berezin transform on line bundles over bounded symmetric domains},
J. Lie Theory \textbf{10} (2000), no.~1, 111--126.

\bibitem {ZhangRec95}G.~Zhang, \emph{Some recurrence formulas for spherical
polynomials on tube domains}, Trans. Amer. Math. Soc. \textbf{347} (1995),
no.~5, 1725--1734.

\bibitem {QZhang17}Q.~Zhang, \emph{On the rigidity and boundary regularity for
Bakry--Emery--Kohn harmonic functions in Bergman metric on the unit ball in
$\mathbb{C}^{n}$}, Proc. Amer. Math. Soc. \textbf{145} (2017), no.~7,
2971--2979.
\end{thebibliography}
\end{document}